\documentclass[smallextended,runningheads]{svjour3}
\smartqed

\usepackage{amsmath}
\usepackage{amssymb}
\usepackage[utf8]{inputenc}
\usepackage{leftindex}
\usepackage{color}
\usepackage{url}

\newcommand{\RR}{\mathbb R}
\newcommand{\CC}{\mathbb C}

\newcommand{\ch}[1]{{\color{blue}#1}}  
\renewcommand{\ch}[1]{#1} 

\newtheorem{remrk}[theorem]{Remark}
\newtheorem{summary}[theorem]{Summary}

\def\rank{\mathop{\rm rank}\nolimits}
\def\nrank{\mathop{\rm nrank}\nolimits}

\def\makeheadbox{}
\journalname{}

\begin{document}

\titlerunning{Analysis of randomized numerical methods for singular matrix pencils}
\title{Analysis of \ch{eigenvalue condition numbers for} a class of randomized numerical methods for singular matrix pencils\thanks{The work of the
first author was supported by the SNSF research project \emph{Probabilistic methods for joint and singular
eigenvalue problems}, grant number: 200021L\_192049. The work of the second author was supported by the
Slovenian Research and Innovation Agency (grants N1-0154\ch{, P1-0294}).}}
\author{Daniel Kressner \and  Bor Plestenjak}
\institute{D. Kressner \at
    Institute of Mathematics, EPFL, CH-1015 Lausanne, Switzerland\\
    \email{daniel.kressner@epfl.ch}
\and B. Plestenjak \at
    Faculty of Mathematics and Physics, University of Ljubljana\ch{, and
    Institute of Mathematics, Physics and Mechanics}, Jadranska 19,
    SI-1000 Ljubljana, Slovenia\\
    \email{bor.plestenjak@fmf.uni-lj.si}
}

\date{}

\maketitle
\begin{abstract}
The numerical solution of the generalized eigenvalue problem for a singular matrix pencil is challenging
due to the discontinuity of its eigenvalues. Classically, such problems are addressed by first extracting the
regular part through the staircase form and then applying a standard solver, such as the QZ
algorithm\ch{, to that regular part}.
Recently, several novel approaches have been proposed to transform the singular pencil into a regular pencil
by relatively simple randomized modifications. In this work, we analyze three such methods by Hochstenbach,
Mehl, and Plestenjak that modify, project, or augment the pencil using random matrices. All three methods rely
on the normal rank and do not alter the finite eigenvalues of the original pencil.
We show that the eigenvalue condition numbers of the transformed pencils are unlikely to be
much larger than the $\delta$-weak eigenvalue condition numbers, introduced by Lotz and Noferini, of the
original pencil. This not only indicates favorable numerical stability but also \ch{reconfirms} that these condition
numbers are a reliable criterion for detecting \ch{simple} finite eigenvalues. We also provide evidence that, from a
numerical stability perspective, the use of complex instead of real random matrices is preferable even for
real singular matrix pencils and real eigenvalues.
\ch{As a side result, we provide sharp left tail bounds for a product of two independent random variables
distributed with the generalized beta distribution of the first kind or Kumaraswamy distribution.}\\

\keywords{Singular pencil \and Singular generalized eigenvalue problem \and Eigenvalue condition number \and
Randomized numerical method \and Random matrices}
\subclass{65F15 \and 15A18 \and 15A22 \and 15A21 \and 47A55 \and 68W20 \and 15B52}
\end{abstract}

\section{Introduction}

The purpose of this work is to study three recent numerical methods, introduced in \cite{HMP19,HMP22}, for computing finite
eigenvalues of a square singular matrix pencil $A-\lambda B$, that is, $A,B\in\CC^{n\times n}$ and $\det(A-\lambda B)\equiv 0$.
We say that $\lambda_0\in\CC$ is an \emph{eigenvalue} of $A-\lambda B$ if $\textrm{rank}(A-\lambda_0 B) < \textrm{nrank}(A,B)$, where
\[
\nrank(A,B) := \max_{\zeta \in \CC} \textrm{rank}(A-\zeta B)<n
\]
is \emph{the normal rank} of the pencil. Similarly, if $\textrm{rank}(B) < \textrm{nrank}(A,B)$
then we say that $A-\lambda B$ has infinite eigenvalue(s).

A major difficulty when working with a singular pencil numerically is the discontinuity of its eigenvalues, that is, the
existence of (arbitrarily small) perturbations of $A,B$ that completely destroy \ch{the} eigenvalue accuracy.
To circumvent this
phenomenon, it is common to first extract the regular part from the staircase form \ch{of the pencil}~\cite{VD79} before applying the QZ
algorithm~\cite{Moler1973,Kragstrom2006} to compute \ch{the} eigenvalues. Notably, the popular software GUPTRI~\cite{Demmel1993,Demmel1993a}
is based on this approach. Numerically, the computation of the staircase form requires several rank decisions and these decisions
tend to become increasingly difficult as the algorithm proceeds, which can ultimately lead to a failure of correctly identifying
and extracting the regular part~\cite{EM,MPl14}.

Despite the discontinuity of \ch{the} eigenvalues mentioned above, Wilkinson~\cite{wilkinson1979kronecker} observed
that the QZ
algorithm directly applied to the original singular pencil usually returns  the eigenvalues of the regular part with
reasonable accuracy. De Ter\'{a}n, Dopico, and Moro~\cite{DM0D8} explained this phenomenon by developing a perturbation
theory for singular pencils, implying that the set of perturbation directions causing discontinuous eigenvalue changes
has measure zero. Later on, Lotz and Noferini~\cite{LotzNoferini} turned this theory into quantitative statements, by
measuring the set of perturbations leading to large eigenvalue changes and defining the notion of weak eigenvalue condition
number for singular pencils.

It is important to note that Wilkinson's observation does not immediately lead to a practical algorithm, because the (approximate)
eigenvalues returned by the QZ algorithm are mixed with the spurious eigenvalues originating from the (perturbed) singular part
and it is a nontrivial task to distinguish these two sets. During the last few years, several methods have been proposed to
circumvent this difficulty.

Inspired by the findings in~\cite{DM0D8}, Hochstenbach, Mehl, and Plestenjak~\cite{HMP19} proposed to introduce a  modification of the form
\begin{equation} \label{eq:HMPpert}
    \widetilde A-\lambda \widetilde B := A-\lambda B+\tau \, (UD_AV^*-\lambda \, UD_BV^*)
\end{equation}
for matrices $D_A,D_B\in\CC^{k\times k}$ with $k:= n-\textrm{nrank}(A,B)$, random matrices $U,V\in \CC^{n\times k}$, and a scalar
$\tau \not= 0$. Generically, $\widetilde A-\lambda \widetilde B$ is a regular pencil and the regular part of $A-\lambda B$ is exactly
preserved\ch{, i.e., if $\lambda_i$ is a finite eigenvalue of $A-\lambda B$ then
$\lambda_i$ is an eigenvalue of $\widetilde A-\lambda \widetilde B$ with the same partial multiplicities.} More specifically, $\lambda_i$ is an eigenvalue of $A-\lambda B$ if and only if $\lambda_i$ is an eigenvalue of
$\widetilde A-\lambda \widetilde B$ such that its right/left  eigenvectors $x$, $y$ satisfy $V^*x=0$ and $U^*y=0$.
The latter property is used to extract the eigenvalues of $A-\lambda B$ from the computed eigenvalues of $\widetilde A-\lambda \widetilde B$.

In \cite{HMP22}, two different variations of the approach from~\cite{HMP19} described above are proposed. Instead of
adding a modification, the pencil is projected to the generically regular pencil $U_\perp^*AV_\perp-\lambda U_\perp^*BV_\perp$
for random matrices $U_\perp,V_\perp\in \CC^{n\times (n-k)}$, and the eigenvalues of $A-\lambda B $ are extracted from the
computed eigenvalues of the smaller pencil. The third method analyzed in this work consists of computing the eigenvalues of
$A-\lambda B$ from the augmented generically regular pencil
\[
    \left[ \begin{array}{cc} A & UT_A \\ S_AV^* & 0 \end{array} \right]
    -
    \lambda \, \left[ \begin{array}{cc} B & UT_B \\ S_BV^* & 0 \end{array} \right],
\]
where $S_A,S_B,T_A,T_B\in\CC^{k\times k}$ and $U,V\in \CC^{n\times k}$ are random matrices. For both variants, it holds
generically that the regular part of $A-\lambda B$ is fully preserved in exact arithmetic.
\ch{Note that roundoff error affects this eigenvalue preservation property when forming the modified pencils~\eqref{eq:HMPpert} and $U_\perp^*AV_\perp-\lambda U_\perp^*BV_\perp$ in finite precision.}

One goal of this work is to show that the modifications introduced by the three methods above are numerically safe. More specifically,
we show that, with high probability, the eigenvalue condition numbers of the modified pencils are not much larger than the weak
eigenvalue condition numbers of the original pencil. In particular, these methods can be expected to return good accuracy for
well-conditioned eigenvalues of $A-\lambda B$ in the presence of roundoff error. Another implication of our result is that the
eigenvalue condition numbers of the modified pencils represent a reliable complementary criterion for identifying reasonably
well-conditioned finite eigenvalues in any of the algorithms from \cite{HMP19} and \cite{HMP22}.

\begin{paragraph}{Related work}
Closer to the analyses in~\cite{DM0D8,LotzNoferini}, it was recently suggested in~\cite{KreGlib22} to perturb the full pencil:
$A+\tau E - \lambda (B+\tau F)$, where $E,F\in\CC^{n\times n}$ are random Gaussian matrices and $\tau>0$ is small but well above
the level of machine precision. Unlike for the three methods mentioned above, the regular part of $A-\lambda B$ is not preserved
by this perturbation. On the other hand, the direct connection to~\cite{LotzNoferini} allows to facilitate their analysis and use
the computed eigenvalue condition numbers of the perturbed pencil as a criterion to identify finite eigenvalues of the original pencil.
In~\cite[P. 2]{KreGlib22}, it was stated that a similar analysis would be more difficult for the method from~\cite{HMP19} because of
the structure imposed on the random perturbation in~\eqref{eq:HMPpert}. In this work, we will address this question and carry over
the analysis from~\cite{KreGlib22,LotzNoferini} to the three methods above. In particular, our analysis confirms that the computed
eigenvalue condition numbers can be used as a reliable indicator for such algorithms as well.
\end{paragraph}

\begin{paragraph}{Outline}
 The structure of the paper is as follows. In Section~\ref{sec:weak_cnd} we review  basic concepts for singular pencils as well
 as $\delta$-weak condition numbers. In Section~\ref{sec:num_methods} we present the three randomized numerical methods that we
 analyze in Section~\ref{sec:analysis}\ch{, where we also obtain the new left tail bounds}. This is followed by numerical examples in  Section~\ref{sec:numresults}. \ch{In the appendix we provide results obtained with symbolic computation that verify the results from
 latter sections.}
\end{paragraph}

\section{Preliminaries}\label{sec:weak_cnd}

\subsection{Reducing subspaces and eigenvectors}

In order to define eigenvectors of a singular pencil according to~\cite{DM0D8,LotzNoferini}, we first introduce the Kronecker
canonical form (KCF) and the notion of minimal reducing subspaces, see, e.g.,~\cite{Gan59,VD83}.

\begin{theorem}[Kronecker canonical form]\label{thm:knf}
Let $A,B\in \CC^{n\times n}$. Then there exist nonsingular matrices $P, Q\in \CC^{n\times n}$ such that
\begin{equation}\label{eq:knf}
  P\,(A-\lambda B)\,Q=\left[\begin{array}{cc} R(\lambda)&0\\ 0 & S(\lambda)\end{array}\right], \qquad
  R(\lambda)=\left[\begin{array}{cc}J-\lambda I_r&0\\ 0&I_s-\lambda N\end{array}\right],
\end{equation}
where $J$ and $N$ are in Jordan canonical form with $N$ nilpotent. Furthermore,
\[
  S(\lambda)=
  {\rm diag}\big(L_{m_1}(\lambda), \dots, L_{m_k}(\lambda), \ L_{n_{1}}(\lambda)^T, \dots, L_{n_{k}}(\lambda)^T\big),
\]
where $L_{j}(\lambda)=[0 \ \, I_{j}]-\lambda \, [I_{j} \ \, 0]$ is of size $j\times (j+1)$, and
$m_i, n_i\ge 0$ for $i=1,\dots,k$, where $k=n-\nrank(A,B)$.
\end{theorem}

The pencil $R(\lambda)$ in~\eqref{eq:knf} is called \emph{the regular part} of $A-\lambda B$ and contains the eigenvalues of
$A-\lambda B$\ch{, where \emph{the Jordan part} $J-\lambda I_r$ contains \emph{the Jordan blocks} of the finite eigenvalues of
$A-\lambda B$. If $\lambda_0$ is a finite eigenvalue of $A-\lambda B$, then its \emph{partial multiplicities} are
the sizes of the Jordan blocks associated with $\lambda_0$.}
A finite eigenvalue is called simple if it is a simple root of $\det R(\lambda)$. The pencil $S(\lambda)$ is called
\emph{the singular part} of $A-\lambda B$ and contains \emph{right singular blocks} $L_{m_1}(\lambda),\dots,L_{m_k}(\lambda)$
and \emph{left singular blocks} $L_{n_{1}}(\lambda)^T,\dots,L_{n_{k}}(\lambda)^T$, where $m_1,\dots,m_k$ and
$n_{1},\dots,n_{k}$ are called the \emph{right} and \emph{left minimal indices} of the pencil, respectively.

We say that a subspace ${\cal M}$ is a \emph{reducing subspace} \cite{VD83} for the pencil $A-\lambda B$ if
${\rm dim}(A{\cal M} + B{\cal M}) = {\rm dim}({\cal M}) - k$, where $k=n-\nrank(A,B)$ counts the number of right singular blocks.
\emph{The minimal reducing subspace} ${\cal M}_{\rm RS}(A,B)$ is the intersection of all reducing subspaces and is spanned
by the columns of $Q$ corresponding to the blocks $L_{m_1}(\lambda),\ldots,L_{m_k}(\lambda)$. Analogously, ${\cal L}$ is a
\emph{left reducing subspace} for the pencil $A-\lambda B$ if ${\rm dim}(A^*{\cal L}+B^*{\cal L})={\rm dim}({\cal L}) - k$ and
\emph{the minimal left reducing subspace} ${\cal L}_{\rm RS}(A,B)$ is the intersection of all left reducing subspaces.

For an eigenvalue $\lambda_0\in\CC$ of $A-\lambda B$, a nonzero vector $x\in\CC^n$ is called a \emph{right eigenvector}
if $(A-\lambda_0 B)x=0$ and $x\not\in {\cal M}_{\rm RS}(A,B)$. A nonzero vector $y\in\CC^n$ such that $y^*(A-\lambda_0 B)=0$
and $y\not\in {\cal L}_{\rm RS}(A,B)$ is called a \emph{left eigenvector}. This agrees with the definition of eigenvectors from
\cite{DopNof,LotzNoferini}. Compared to a regular pencil, the eigenvectors of singular pencils have a much larger degree of
non-uniqueness, due to components from the minimal reducing subspaces.

\subsection{Eigenvalue perturbation theory}

Suppose that we perturb an $n\times n$ matrix pencil $A-\lambda B$ into
\begin{equation}\label{eq:pertAB}
    \widetilde A-\lambda \widetilde B:=A+\epsilon E-\lambda (B+\epsilon F),
\end{equation}
where $\epsilon>0$.
We define $\|(E,F)\|:=(\|E\|_F^2+\|F\|_F^2)^{1/2}$, where $\|\cdot\|_F$ denotes the Frobenius norm of a matrix. When
$\|(E,F)\|=1$ we can identify the pencil $E-\lambda F$ with a point on the unit sphere in $\CC^{2n^2}$ and think of
$(E,F)$ as a direction of the perturbation~\eqref{eq:pertAB}.

Before addressing the singular case, let us first recall the classical eigenvalue perturbation theory~\cite{StewartSun}
for a \emph{regular} pencil $A-\lambda B$. Consider a simple \ch{finite} eigenvalue $\lambda_0\in\CC$ of $A-\lambda B$ with normalized
right/left eigenvectors $x$, $y$. For $\|(E,F)\|=1$ and sufficiently small $\epsilon>0$ there exists an eigenvalue
$\lambda_0(\epsilon)$ of the perturbed pencil \eqref{eq:pertAB} satisfying the perturbation expansion
\begin{equation}\label{eq:pert_reg_case}
  |\lambda_0(\epsilon)-\lambda_0|=
  \frac{|y^*(E-\lambda_0 F)x|}{|y^*Bx|}\epsilon +{\cal O}(\epsilon^2)
  \le \frac{(1+|\lambda_0|^2)^{1/2}}{|y^*Bx|}\epsilon +{\cal O}(\epsilon^2)
\end{equation}
\ch{as $\epsilon \to 0$.} 
Note that the inequality becomes an equality for the direction $E=(1+|\lambda_0|^2)^{-1/2}yx^*$, $F=-\overline{\lambda}_0E$.
In turn, the \emph{absolute condition number} of $\lambda_0$, defined as
\begin{equation}\label{eq:abs_cond_num}
  \kappa(\lambda_0)=\lim_{\epsilon\to 0}\sup_{\|(E,F)\|\le 1}
  \ch{\frac{1}{\epsilon}}|\lambda_0(\epsilon)-\lambda_0|,
\end{equation}
satisfies
\begin{equation}\label{eq:kappa1}
  \kappa(\lambda_0)=1 / \gamma(\lambda_0),\ {\rm where}\ \gamma(\lambda_0)=|y^*Bx|(1+|\lambda_0|^2)^{-1/2};
\end{equation}
see, e.g., \cite[Lemma 3.1 and Eq. (3.3)]{Fraysse}.

For a singular pencil, the definition~\eqref{eq:abs_cond_num} always leads to an infinite condition number because of the
discontinuity of eigenvalues. To address this, we first recall the eigenvalue expansion by De Ter\'an, Dopico, and
Moro~\cite[Corollary 2]{DM0D8}.

\begin{theorem}\label{thm:DeTeranDopico}
Let $\lambda_0$ be a \ch{finite} simple eigenvalue of an $n\times n$ pencil $A-\lambda B$ of normal rank $n-k$, $k\ge 1$. Let
$X=[X_1\ x]$, $Y=[Y_1\ y]$ be $n\times (k+1)$ matrices with orthonormal columns such that: $X_1$ is a basis for
${\rm ker}(A-\lambda_0 B)\cap {\cal M}_{\rm RS}(A,B)$, $X$ is a basis for ${\rm ker}(A-\lambda_0 B)$, $Y_1$
is a basis for ${\rm ker}((A-\lambda_0 B)^*)\cap {\cal L}_{\rm RS}(A,B)$, and $Y$ is a basis for ${\rm ker}((A-\lambda_0 B)^*)$.
If $E-\lambda F$ is such that \ch{$\det(Y_1^*(E-\lambda_0 F)X_1)\ne 0$},
then, for sufficiently small $\epsilon>0$,
there exists an eigenvalue $\lambda_0(\epsilon)$ of the perturbed pencil \eqref{eq:pertAB} such that
\begin{equation}\label{eq:ocena_det_det}
  \lambda_0(\epsilon)=\lambda_0 - \frac{\det(Y^*(E-\lambda_0 F)X)}
  {y^*Bx\cdot\det(Y_1^*(E-\lambda_0 F)X_1)}\epsilon +{\cal O}(\epsilon^2).
\end{equation}
\end{theorem}

The expansion~\eqref{eq:ocena_det_det} allows one to define \emph{the directional sensitivity} for perturbations
$E-\lambda F$, $\|(E,F)\|=1$, satisfying the condition of Theorem~\ref{thm:DeTeranDopico}:
\begin{equation}\label{eq:dir_sens}
  \sigma_{E,F}(\lambda_0)=\left|\frac{\det(Y^*(E-\lambda_0 F)X)}
  {y^*Bx\cdot\det(Y_1^*(E-\lambda_0 F)X_1)}\right|;
\end{equation}
see~\cite[Definition 2.4 and Corollary 3.3]{LotzNoferini}.
\ch{We can now generalize the definition of $\gamma(\lambda_0)$ in~\eqref{eq:kappa1} to singular pencils by using
the right/left eigenvectors $x$, $y$ from Theorem~\ref{thm:DeTeranDopico}.}
Because of the appearance of the factor $y^*Bx$ in the denominator of~\eqref{eq:dir_sens}, we expect that the
quantity $1/\gamma(\lambda_0)$ continues to play a crucial role in determining the sensitivity
of an eigenvalue. However, it is important to note that $x,y$ are particular choices of eigenvectors made in
Theorem~\ref{thm:DeTeranDopico}: the right eigenvector $x$ is orthogonal to the right minimal reducing subspace
${\cal M}_{\rm RS}$ and the left eigenvector $y$ is orthogonal to ${\cal L}_{\rm MR}$. Under these constraints, the
eigenvectors $x$ and $y$ of the simple eigenvalue $\lambda_0$ become uniquely determined up to multiplication by unit complex numbers.
It follows from $Y_1^* B X_1 = 0$\ch{, $Y_1^*Bx=0$ and $y^*BX_1=0$} that this particular choice maximizes $|y^*B x|$, i.e., for all vectors $z$ and $w$ such that
$\|z\|_2=\|w\|_2=1$, $(A-\lambda_0 B)z=0$, and $w^*(A-\lambda_0 B)=0$, it holds that
\begin{equation}\label{eq:ocena_gamma}
  |w^*Bz|(1+|\lambda_0|^2)^{-1/2} \le \gamma(\lambda_0).
\end{equation}

\begin{remrk}\label{rem:real_complex}
If $A,B$ are real matrices and $\lambda_0$ is a real simple eigenvalue, then the matrices $X,X_1,Y,Y_1$ and vectors $x,y$ in
Theorem~\ref{thm:DeTeranDopico} can be chosen to be real as well.
\end{remrk}

Clearly, the perturbations for 
\ch{which the quantity
$\det(Y_1^*(E-\lambda_0 F)X_1)$ in the denominator of~\eqref{eq:ocena_det_det} vanishes} form a set of measure zero in the
unit sphere in $\CC^{2n^2}$. In other words, this event has zero probability if we draw $(E,F)$ uniformly at random from the
the unit sphere in $\CC^{2n^2}$, which we will denote by $(E,F)\sim {\cal U}(2n^2)$.
The  \emph{$\delta$-weak condition number} introduced by Lotz and Noferini~\cite[Definition 2.5]{LotzNoferini} offers a more
refined picture by measuring the tightest upper bound $t$ such that the directional sensitivity~\eqref{eq:dir_sens} stays
below $t$ with probability at least $1-\delta$.

\begin{definition}\label{def:delta-weak}
Let $\lambda_0\in\CC$ be \ch{a finite simple} eigenvalue of a singular pencil $A-\lambda B$. The
\emph{$\delta$-weak condition number} of $\lambda_0$ is defined as
\[
  \kappa_w(\lambda_0;\delta)=\inf\big\{t\in \RR:\ \mathbb P(\sigma_{E,F}(\lambda_0)<t)\ge 1-\delta\big\},\quad (E,F)\sim {\cal U}(2n^2).
\]
\end{definition}

\ch{If $\lambda_0$ is a finite simple eigenvalue of a regular pencil $A-\lambda B$, it follows from \eqref{eq:pertAB}
and \eqref{eq:kappa1} that
$\sigma_{E,F}(\lambda_0)=|y^*(E-\lambda_0 F)x|/|y^*Bx|\le \kappa(\lambda_0)=1/\gamma(\lambda_0)$
for all $\|(E,F)\|=1$. Therefore, if we apply Definition~\ref{def:delta-weak} to
a regular pencil, it follows from \eqref{eq:abs_cond_num} that}
$\kappa_w(\lambda_0;\delta)$
converges to $\kappa(\lambda_0) = 1/\gamma(\lambda_0)$
monotonically from below as $\delta \downarrow 0$. The following result from~\cite[Theorem 5.1]{LotzNoferini}
and~\cite[Theorem 3.1]{KreGlib22} suggests to use $1/\gamma(\lambda_0)$ as a proxy for eigenvalue sensitivity in the singular case as well.
\begin{theorem} \label{theorem:sensitvity}
Let $\lambda_0\in\CC$ be \ch{a finite simple} eigenvalue of an $n\times n$ singular pencil $A-\lambda B$ of normal
rank $n-k$. Then for $\delta \le k / (2n^2)$ it holds that
\[
  \frac{1}{\sqrt{\delta 2n^2}} \cdot \frac{1}{\gamma(\lambda_0)} \le \kappa_w(\lambda_0;\delta)\le  \frac{\sqrt{k}}{\sqrt{\delta 2n^2}} \cdot \frac{1}{\gamma(\lambda_0)},
\]
where $\gamma(\lambda_0)$ is defined as in \eqref{eq:kappa1} with the right/left eigenvectors $x$, $y$ from Theorem~\ref{thm:DeTeranDopico}.
\end{theorem}
The algorithms for the singular generalized eigenvalue problem presented in the next section use $\gamma(\lambda_0)$ to
identify (finite) eigenvalues numerically.

\ch{We finish this section by remarking that the absolute condition number~\eqref{eq:abs_cond_num}
is consistent with the definitions in~\cite{LotzNoferini,KreGlib22}. Following, e.g.,~\cite{HighamHigham97} one could also consider a notion of relative condition number that imposes $\|E\|_F\le \|A\|_F$ and $\|F\|_F\le \|B\|_F$ instead of $\|(E,F)\|\le 1$ on the perturbation direction in~\eqref{eq:pertAB}. In effect, both the standard and weak (absolute) condition numbers get multiplied by the factor $(\|A\|_F^2+\|B\|_F^2)^{1/2}$. Under reasonable choices of the parameters involved, this additional factor does not differ significantly for the modified pencils. In particular, our results presented for absolute condition numbers easily extend to relative condition numbers.
}

\section{Randomized numerical methods based on the normal rank} \label{sec:num_methods}

In this section, we describe in some detail the three numerical methods from~\cite{HMP19} and \cite{HMP22} for computing the finite
eigenvalues of an $n\times n$ singular pencil $A-\lambda B$ with $\nrank(A,B)=n-k$ for $k \ge 1$. All three methods require knowledge
about the exact normal rank in order to leave the regular part intact. If this quantity is not known a priori, it can be determined
from $\rank(A-\xi_i B)$ for a small number of randomly chosen $\xi_i\in\CC$.\footnote{Although this heuristics works very well in
practice, an analysis is beyond the scope of this work. Already the seemingly  simpler special case of numerically deciding whether
$A-\lambda B$ is close to a singular pencil is quite intricate~\cite{KressnerVoigt2015}. }

\subsection{Rank-completing modification} \label{sec:rank-complete}

We first consider the rank-completing method from \cite{HMP19}, where a random pencil of normal rank $k$ is added to yield
a (generically regular) matrix pencil
\begin{equation}\label{pert}
    \widetilde A-\lambda \widetilde B := A-\lambda B+\tau \, (UD_AV^*-\lambda \, UD_BV^*),
\end{equation}
where $D_A,D_B\in\CC^{k\times k}$ are diagonal matrices such that $D_A-\lambda D_B$ is regular, $U,V\in \CC^{n\times k}$ are
matrices of rank $k$, and $\tau\in\CC$ is nonzero. Note that $k$ is the smallest normal rank for such a modification to turn a
singular into a regular pencil. The following result \cite[Summary~4.7]{HMP19}, see also \cite[Remark 3.5]{HMP22}, characterizes
the dependence of eigenvalues and eigenvectors of the modified pencil~\eqref{pert} on $\tau$, $D_A$, $D_B$, $U$, and $V^*$.

\begin{summary} \rm \label{summry}
Let $A-\lambda B$ be an $n\times n$ singular pencil of normal rank $n-k$ with left minimal indices $n_1,\dots,n_k$ and right minimal
indices $m_1,\dots,m_k$. Let $N = n_1 + \cdots + n_k$ and $M = m_1 + \cdots + m_k$. Then the regular part of $A-\lambda B$ has size
$r:=n-N-M-k$ and generically (with respect to the entries of $D_A,D_B,U,V^*$), the modified pencil $\widetilde A-\lambda \widetilde B$
defined in~\eqref{pert} is regular and its eigenvalues are classified in the following four disjoint groups:
\begin{enumerate}
\item \emph{True eigenvalues}: There are $r$ eigenvalues \ch{counted with their multiplicities} that coincide with the eigenvalues of the original pencil $A-\lambda B$ \ch{with the same partial multiplicities}.
The right/left eigenvectors $x,y$ of $\widetilde A-\lambda \widetilde B$ belonging to these eigenvalues satisfy the orthogonality
relations $V^*x=0$ and $U^*y=0$.
\item \emph{Prescribed eigenvalues}: There are $k$ eigenvalues such that $V^*x\ne 0$ for all right eigenvectors $x$ and
$U^*y\ne 0$ for all left eigenvectors $y$. These are the $k$ eigenvalues of $D_A-\lambda D_B$.
\item \emph{Random right eigenvalues}: There are $M$ eigenvalues, which are all simple and such that $V^*x=0$ for all right
eigenvectors $x$ and $U^*y\ne 0$ for all left eigenvectors $y$.
\item \emph{Random left eigenvalues}: There are $N$ eigenvalues, which are all simple and such that $V^*x\ne 0$ for all right
eigenvectors $x$ and $U^*y=0$ for all left eigenvectors $y$.
\end{enumerate}
\end{summary}
\smallskip

Summary~\ref{summry} has the following practical consequences. If we compute all eigenvalues $\lambda_i$ of~\eqref{pert}, together
with the (normalized) right and left  eigenvectors $x_i$ and $y_i$ for $i=1,\ldots,n$, then $\max(\|V^*x_i\|_2,\|U^*y_i\|_2)=0$
if and only if $\lambda_i$ is an eigenvalue \ch{of $A - \lambda B$}. In numerical computations, we can use
$\max(\|V^*x_i\|_2,\|U^*y_i\|_2)<\delta_1$, where $\delta_1$ is a prescribed threshold, as a criterion to extract the true eigenvalues
in the first phase. Note that \ch{for a simple finite eigenvalue} $x_i$ and $y_i$ are unique (up to multiplication by unit complex numbers) because, \emph{generically},
the {\em modified} pencil is regular. They correspond to eigenvectors of the original singular
pencil satisfying the orthogonality constraints $V^*x_i =0$ and $U^*y_i =0$.

In the second phase, we use the (reciprocal) eigenvalue sensitivities for extracting \ch{simple} finite eigenvalues, that is, we compute
\begin{equation}\label{eq:gamma_i}
  \gamma_i=|y_{i}^*Bx_i|(1+|\lambda_i|^2)^{-1/2}
\end{equation}
and identify $\lambda_i$ as a finite eigenvalue if $\gamma_i>\delta_2$ for a prescribed threshold $\delta_2$. Note that $1/\gamma_i$
is the absolute condition number of $\lambda_i$ as an eigenvalue of the (generically) regular pencil~\eqref{pert}; see~\eqref{eq:kappa1}.

For different matrices $U$ and $V$ in~\eqref{pert} we obtain different eigenvectors $x_i$ and $y_i$ and thus different values
of $\gamma_i$ for the same eigenvalue\ch{, while the changing of $\tau$, $D_A$ and $D_B$ does not affect the eigenvectors
\cite[Lemma 3.4]{HMP22}}. In Section~\ref{sec:analysis} we will analyze these values for random $U$ and $V$ and compare
them to the unique value $\gamma(\lambda_i)$ that appears in the $\delta$-weak condition number of $\lambda_i$ as an
eigenvalue of the singular pencil $A-\lambda B$; see Theorem~\ref{theorem:sensitvity}.

The considerations above lead to Algorithm~1 from~\cite{HMP19}.
In theory, the results returned by the algorithm are
independent of $\tau \ne 0$. In practice, $|\tau|$ should be neither too small nor too large in order to limit the impact of roundoff error;
in \cite{HMP19} it is suggested to \ch{scale $A$ and $B$ so that
$\|A\|_1=\|B\|_1=1$, where $\|A\|_1=\max_{1\le j\le n}\sum_{i=1}^n|a_{ij}|$, and}
take $\tau=10^{-2}$. The quantity $\varepsilon$ stands for the machine precision.

\noindent\vrule height 0pt depth 0.5pt width \textwidth \\
{\bf Algorithm~1: Eigenvalues of singular pencil by rank-completing modification}. \\[-3mm]
\vrule height 0pt depth 0.3pt width \textwidth \\
{\bf Input:} $A,B\in\CC^{n\times n}$ such that $\|A\|_1=\|B\|_1=1$, $k=n-\nrank(A,B)$, parameter $\tau$ (default $10^{-2}$),
    thresholds $\delta_1$ (default $\varepsilon^{1/2}$) and $\delta_2$ (default $10^{2}\,\varepsilon$).\\
{\bf Output:} \ch{Simple} finite eigenvalues of $A-\lambda B$. \\
\begin{tabular}{ll}
{\footnotesize 1:} & Select random $n\times k$ matrices $U$ and $V$ with orthonormal columns. \\
{\footnotesize 2:} & Select random diagonal $k\times k$ matrices $D_A$ and $D_B$.\\
{\footnotesize 3:} & Compute the eigenvalues $\lambda_i$, $i=1,\ldots,n$, and right and left normalized\\
    & eigenvectors $x_i$ and $y_i$ of the perturbed pencil \eqref{pert}.\\
{\footnotesize 4:} & Compute $\gamma_i=|y_i^*\ch{B}x_i|(1+|\lambda_i|^2)^{-1/2}$ for $i=1,\ldots,n$.\\
{\footnotesize 5:} & Compute $\sigma_i=\|V^*x_i\|_2$, $\tau_i=\|U^*y_i\|_2$ for $i=1,\ldots,n$.\\
{\footnotesize 6:} & Return eigenvalues $\lambda_i$, $i=1,\ldots,n$, for which ${\rm max}(\sigma_i,\tau_i)<\delta_1$ and\\
 & $\gamma_i>\delta_2$.
\end{tabular} \\
\vrule height 0pt depth 0.5pt width \textwidth
\medskip

By the theory in~\cite{HMP19} and~\cite{HMP22}, the eigenvectors of the modified pencil~\eqref{pert} that correspond to true
eigenvalues do not change if we replace $U$ and $V$  by $\widetilde U=UR$ and $\widetilde V=VS$, where $R$ and $S$ are arbitrary
nonsingular $k\times k$ matrices. Thus, choosing $U,V$ to have orthonormal columns does not violate the genericity assumption
in Summary~\ref{summry}.

\ch{Note that $\gamma_i$ in line 4 was initially computed in \cite{HMP19} as
$\gamma_i=|y_i^*\widetilde B x_i|$. This was changed to
$\gamma_i=|y_i^*\widetilde B x_i|(1+|\lambda_i|^2)^{-1/2}$ in \cite{HMP22} to be consistent with \cite{LotzNoferini} and \cite{KreGlib22}.
Since for true eigenvalues $y_i^*\widetilde B x_i=y_i^*B x_i$, we use $B$ instead of $\widetilde B$ to simplify the analysis.}

We remark that \ch{the} values $\gamma_i$ from~\eqref{eq:gamma_i} were also used in~\cite{KreGlib22} for computing
finite eigenvalues of a
singular pencil via unstructured random perturbations. The use of full-rank perturbations comes with two disadvantages: the orthogonality
relations for the eigenvectors exploited above are not satisfied and, in contrast to Algorithm~1, the eigenvalues of the perturbed
pencil in~\cite{KreGlib22} differ from the exact eigenvalues of the original pencil. The latter leads one to choose $\tau > 0$
very small, but at the same time it needs to stay well above the level of machine precision.

\ch{Let us also remark that Algorithm~1, with additional heuristic criteria,
can in practice, due to a "positive" effect of roundoff error, successfully compute multiple finite eigenvalues as well,
for details see \cite[Sec. 6]{HMP22}.}

\subsection{Normal rank projections}

In~\cite{HMP22}, a variant of Algorithm~1 was proposed that uses random projections to a pencil of smaller size, equal
to the normal rank $n-k$. In addition, the method does not require to choose the matrices $D_A$, $D_B$, and the parameter $\tau$.

For $U_\perp, V_\perp \in \mathbb C^{n \times (n-k)}$, we consider the $(n-k)\times (n-k)$ pencil $U_{\perp}^*(A-\lambda B)\,V_{\perp}$.
To connect it to the modified pencil~\eqref{pert} used in Algorithm~1, let us assume that the columns of $U_\perp$ and $V_\perp$
span the orthogonal complements of ranges of $U$ and $V$, respectively, so that $U_\perp^* U = 0$,  $V_\perp^* V = 0$. The pencil
\begin{align*}\label{coord}
  \widehat A-\lambda \widehat B&:= [U \ \, U_{\perp}]^*\,(A - \lambda B + \tau(UD_AV^* -\lambda UD_BV^*)\,[V \ \, V_{\perp}] \\
  &= \left[
  \begin{array}{cc}
  U^*(A-\lambda B)V +\tau(D_A-\lambda D_B) & U^*(A-\lambda B)V_{\perp} \\[1mm]
  U_{\perp}^*(A-\lambda B)V & U_{\perp}^*(A-\lambda B)V_{\perp}
  \end{array}
  \right]
\end{align*}
is then equivalent to \eqref{pert} and we observe the following.

\begin{proposition}[{\cite[Proposition 4.1 and Theorem 4.2]{HMP22}}] \label{prop_proj}
Let $A-\lambda B$ be a complex $n\times n$ singular pencil of normal rank $n-k$. Then, under the assumptions of
Summary~\ref{summry}, the $(n-k)\times (n-k)$ pencil $A_{22}-\lambda B_{22}:=U_{\perp}^*(A-\lambda B)\,V_{\perp}$ is
\ch{generically} regular
and the eigenvalues of $A_{22}-\lambda B_{22}$ are precisely:
\begin{enumerate}
\item[a)] the random eigenvalues of \eqref{pert} (groups 3 and 4 in Summary~\ref{summry});
\item[b)] the true eigenvalues of $A-\lambda B$.
\end{enumerate}
\end{proposition}

Based on the above results, an algorithm is devised in \cite{HMP22}. Algorithm~2 is a simplified form that matches
Algorithm~1 as much as possible.

\noindent\vrule height 0pt depth 0.5pt width \textwidth \\
{\bf Algorithm~2: Eigenvalues of singular pencil by normal rank projection}. \\[-3mm]
\vrule height 0pt depth 0.3pt width \textwidth \\
{\bf Input and output:} See Algorithm 1.\\
\begin{tabular}{ll}
{\footnotesize 1:} & Select random unitary $n\times n$ matrices $[U \ \, U_\perp]$ and $[V \ \, V_\perp]$,
    where $U$ \\
    & and $V$ have $k$ columns.\\
{\footnotesize 2:} & Compute the eigenvalues $\lambda_i$, $i=1,\dots,n-k$, and right and left\\
  & normalized eigenvectors $x_i$ and $y_i$ of $U_\perp^*(A-\lambda B)V_\perp$.\\
{\footnotesize 3:} & Compute $\sigma_i=\|U^*(A-\lambda_iB)V_\perp x_i\|_2$, $\tau_i=\|y_i^*U_\perp^*(A-\lambda_iB)V\|_2$ for\\
 & $i=1,\dots,n-k$. \\
{\footnotesize 4:} & Compute $\gamma_i=|y_i^*U_\perp^*BV_\perp x_i|\,(1+|\lambda_i|^2)^{-1/2}$ for $i=1,\ldots, n-k$.\\
{\footnotesize 5:} & Return eigenvalues $\lambda_i$, $i=1,\ldots,n-k$, for which \\
& $\max(\sigma_i,\tau_i) < \delta_1(1+|\lambda_i|)$ and $\gamma_i>\delta_2$.
\end{tabular} \\
\vrule height 0pt depth 0.5pt width \textwidth
\medskip

\noindent The following corollary shows that the reciprocal eigenvalue condition number $\gamma_i$ computed in line 4 of
Algorithm 2 matches the corresponding quantity of Algorithm 1.
\begin{corollary}\label{cor:zveza}
Let $\lambda_i\in\CC$ be a simple eigenvalue of a singular pencil $A-\lambda B$.
Under the assumptions of Proposition~\ref{prop_proj}, if
\begin{enumerate}
    \item[a)] $(\lambda_i,x_i,y_i)$ is an eigentriple of \eqref{pert} such that
    $\|x_i\|_2=1$, $\|y_i\|_2=1$, $V^*x_i=0$, and $U^*y_i=0$; and
    \item[b)] $(\lambda_i,w_i,z_i)$ is an eigentriple of $A_{22}-\lambda B_{22}$ such that
    $\|w_i\|_2=1$, $\|z_i\|_2=1$,
    $U^*(A-\lambda_iB)V_\perp w_i=0$, and $z_i^*U_\perp^*(A-\lambda_iB)V=0$,
\end{enumerate}
then $|y_i^*Bx_i|=|z_i^*U_\perp^*BV_\perp w_i|$.
\end{corollary}
\begin{proof}
Since $\lambda_i$ is simple, the vectors $x_i,y_i$ from a) and $w_i,z_i$ from b) are uniquely defined up to multiplication
by unit complex numbers. If a) and b) both hold then it immediately follows that, up to sign changes, $x_i=V_\perp w_i$ and
$y_i=U_\perp z_i$.
\end{proof}

\subsection{Augmentation} The third method, also presented in \cite{HMP22}, uses the $(n+k) \times (n+k)$ augmented (or bordered)
matrix pencil
\begin{equation} \label{augm1x}
    A_a-\lambda B_a :=
    \left[ \begin{array}{cc} A & UT_A \\ S_AV^* & 0 \end{array} \right]
    -
    \lambda \, \left[ \begin{array}{cc} B & UT_B \\ S_BV^* & 0 \end{array} \right],
\end{equation}
where $S_A,S_B,T_A$, and $T_B$ are $k\times k$ diagonal matrices and $U,V$ are $n\times k$ matrices.

\begin{proposition}[{\cite[Proposition 5.1]{HMP22} \label{prop_augm}}]
Let $A-\lambda B$ be an $n\times n$ singular pencil of normal rank $n-k$ such that all its eigenvalues are semisimple. Assume that the
diagonal $k\times k$ pencils $S_A-\lambda S_B$ and $T_A-\lambda T_B$ are regular and that their $2k$ eigenvalues are pairwise distinct.
Furthermore, let $U,V\in \CC^{n\times k}$ have orthonormal columns such that the augmented pencil \eqref{augm1x} is regular. Then the
pencil \eqref{augm1x} has the following eigenvalues:
\begin{enumerate}
\item[a)] $2k$ prescribed eigenvalues, which are precisely the eigenvalues of $S_A-\lambda S_B$ and $T_A-\lambda T_B$;
\item[b)] the random eigenvalues of \eqref{pert} (groups 3 and 4 in Summary~\ref{summry}) with the same $U$ and $V$ and with $D_A=T_AS_A$, $D_B=T_BS_B$;
\item[c)] the true eigenvalues of $A-\lambda B$.
\end{enumerate}
\end{proposition}

The algorithm based on the above proposition is given in Algorithm~3.

\noindent\vrule height 0pt depth 0.5pt width \textwidth \\
{\bf Algorithm~3: Eigenvalues of singular pencil by augmentation}. \\[-3mm]
\vrule height 0pt depth 0.3pt width \textwidth \\
{\bf Input and output:} See Algorithm~1.\\
\begin{tabular}{ll}
{\footnotesize 1:} & Select random $n\times k$ matrices $U$ and $V$ with orthonormal columns.\\
{\footnotesize 2:} & Select random diagonal $k \times k$ matrices $T_A$, $T_B$, $S_A$, and $S_B$\\
{\footnotesize 3:} & Compute the eigenvalues $\lambda_i$, $i=1,\dots,n+k$, and normalized right and\\
  & left eigenvectors
    $[x_{i1}^T\ x_{i2}^T]^T$ and $[y_{i1}^T\ y_{i2}^T]^T$ of the augmented pencil \eqref{augm1x}.\\
{\footnotesize 4:} & Compute $\sigma_i=\|x_{i2}\|_2$, \ $\tau_i=\|y_{i2}\|_2$, \ $i=1,\dots,n+k$. \\
{\footnotesize 5:} & Compute $\gamma_i=|y_{i1}^*Bx_{i1}| \, (1+|\lambda_i|^2)^{-1/2}$, \ $i=1,\ldots, n+k$.\\
{\footnotesize 6:} & Return all eigenvalues $\lambda_i$, $i=1,\ldots,n+k$, where $\max(\sigma_i,\tau_i) < \delta_1$ and\\
& $\gamma_i>\delta_2$.
\end{tabular} \\
\vrule height 0pt depth 0.5pt width \textwidth
\medskip

\noindent Again, the reciprocal eigenvalue condition number $\gamma_i$ computed in line \ch{5} of Algorithm 3 matches the
corresponding quantity of Algorithm 1.
\begin{corollary}\label{cor:zveza_augm} Under the assumptions of Proposition~\ref{prop_augm}, if
\begin{enumerate}
    \item[a)] $(\lambda_i,x_i,y_i)$ is an eigentriple of \eqref{pert} such that
    $\|x_i\|_2=1$, $\|y_i\|_2=1$, $V^*x_i=0$, and $U^*y_i=0$; and
    \item[b)] $(\lambda_i,w_i,z_i)$, where
    $w_i=[w_{i1}^T\ 0]^T$ and $z_i=[z_{i1}^T\ 0]^T$, such that $w_{i1}\in\CC^n$, $z_{i1}\in\CC^n$,
    $\|w_i\|_2=1$, $\|z_i\|_2=1$, is an eigentriple of the augmented pencil \eqref{augm1x},
\end{enumerate}
then $|y_i^*Bx_i|=|z_i^*B_aw_i|=|z_{i1}^*Bw_{i1}|$.
\end{corollary}
\begin{proof}
If a) and b) are both true then it immediately follows that, up to sign changes,
$x_i=w_{i1}$ and $y_i=z_{i1}$.\qed
\end{proof}

\section{Probabilistic analysis}\label{sec:analysis}

Our goal is to analyze the behavior of the quantities $\gamma_i$ in Algorithms 1--3 and show that they are unlikely to be
much below $\gamma(\lambda_i)$. It follows from Corollaries~\ref{cor:zveza} and \ref{cor:zveza_augm} that it is sufficient
to consider Algorithm~1 and the quantity $\gamma_i$ defined in~\eqref{eq:gamma_i}.

In the following, we assume that $\lambda_i$ is a simple eigenvalue of $A-\lambda B$ and let $X=[X_1\ x]$ and $Y=[Y_1\ y]$
denote the orthonormal bases for ${\rm ker}(A-\lambda_i B)$ and ${\rm ker}((A-\lambda_i B)^*)$ introduced in Theorem
\ref{thm:DeTeranDopico}. We recall from Theorem~\ref{theorem:sensitvity} that the reciprocal of
$\gamma(\lambda_i) = |y^* B x| (1+|\lambda_i|^2)^{-1/2}$ critically determines the sensitivity of $\lambda_i$ as an eigenvalue of
$A-\lambda B$. The sensitivity of $\lambda_i$ as an eigenvalue of
$\widetilde A-\lambda \widetilde B$ from~\eqref{pert} is given by $1/\gamma_i$ with
$\gamma_i = |y_i^* B x_i| (1+|\lambda_i|^2)^{-1/2}$; see~\eqref{eq:gamma_i}.
The eigenvectors $x_i, y_i$ are normalized ($\|x_i\|_2=\|y_i\|_2=1$) and depend on the choices of $U$ and $V$ in Algorithm~1.
Generically (with respect to $D_A,D_B,U,V^*$), Summary~\ref{summry}.1 yields the relations
\begin{equation}\label{eq:alfa_beta}
  x_i=[X_1\ x]\left[\begin{matrix} a \cr \alpha\end{matrix}\right]\quad {\rm and}\quad
  y_i=[Y_1\ y]\left[\begin{matrix} b \cr \beta\end{matrix}\right],\quad V^*x_i=0, \quad U^*y_i=0.
\end{equation}
If $V^* X_1$ is invertible, the choice of $V$ entirely determines $|\alpha|$ because
$V^* x_i = 0$ implies $a = -(V^* X_1)^{-1} V^* x \alpha$ and, hence, $\|x_i\|_2 = 1$ implies
\[|\alpha| = 1/\sqrt{1 + \|(V^* X_1)^{-1} V^* x\|_2^2}.\]
Analogously, the choice of $U$ determines $|\beta|$ if $U^* Y_1$ is invertible.

Since $Y^*BX_1=0$ and $Y_1^*BX=0$, we get $y_i^*Bx_i = \alpha \ch{\overline \beta} y^*Bx$ and thus
\[
  \gamma_i=|\alpha| |\beta| \gamma(\lambda_i).
\]
The relation $0\le |\alpha|,|\beta|\le 1$ immediately gives $\gamma_i\le \gamma(\lambda_i)$, in line with~\eqref{eq:ocena_gamma}.
A small value of $|\alpha||\beta|$ means an increased eigenvalue sensitivity for $\widetilde A-\lambda \widetilde B$, potentially
causing Algorithm~1 to yield unnecessarily inaccurate results. In the following, we will show that this is unlikely when random
matrices $U,V$ are used in Algorithm~1. This also implies that the (reciprocal) condition numbers computed in Algorithm~1 can
be used with high probability to correctly identify finite simple eigenvalues.

\subsection{Preliminary results}

Let ${\cal N}^1(\mu,\ch{\sigma^2})$ denote the normal distribution with mean $\mu$ and \ch{variance $\sigma^2$}. In particular,
$x\sim {\cal N}^1(0,1)$ is a standard (real) normal random variable. We write $z\sim {\cal N}^2(0,1)$ if $z=x+\mathrm{i} y$
is a standard complex normal variable, that is, $x,y\sim {\cal N}^1(0,\frac{1}{2})$ are independent.
In the following, we will analyze real matrices ($\mathbb F=\RR$) and complex matrices ($\mathbb F=\CC$) simultaneously. For this purpose,
we set $\phi=1$ for $\mathbb F=\RR$ and $\phi=2$ for $\mathbb F=\CC$.

The matrices $U$ and $V$ from Algorithm~1 belong to \emph{the Stiefel manifold}
\[\mathbb V_k^n(\mathbb F)=\{Q\in\mathbb F^{n\times k}:\ Q^*Q=I\}.\]
We will choose them randomly (and independently) from the uniform distribution on $\mathbb V_k^n(\mathbb F)$.
A common way to compute such a matrix is to perform the QR decomposition of an $n\times k$
\emph{Gaussian random matrix} $M$, that is, the entries of $M$ are i.i.d.~real or complex standard normal variables,
see e.g., \cite{Mezzadri,Stewart}. That this indeed yields the uniform distribution follows from
the following variant of the well-known Bartlett decomposition theorem (\cite[Theorem 3.2.14]{Muirhead}, \cite[Proposition 7.2]{Eaton}); see also \cite[Proposition 4.5]{LotzNoferini}.

\begin{theorem}\label{thm:bartlett}
For $\mathbb F\in\{\RR,\CC\}$, let $M\in\mathbb F^{n\times k}$, $n \ge k$, be a Gaussian random matrix. Consider the QR decomposition $M=QR$,
  where $Q\in\mathbb V_k^n(\mathbb F)$ and $R\in\mathbb F^{k\times k}$ is upper triangular with non-negative diagonal entries. Then
  \begin{enumerate}
      \item[a)] the entries of $Q$ and the entries of the upper triangular part of $R$ are all independent random variables;
      \item[b)] $Q$ is distributed uniformly over $\mathbb V_k^n(\mathbb F)$;
      \item[c)] $r_{ij}\sim {\cal N}^\phi(0,1)$ for $1\le i< j\le k$;
      \item[d)] $\phi \ch{r_{jj}^2}\sim\chi^2(\phi(n-j+1))$ for $j=1,\ldots,k$;
    \end{enumerate}
  where $\chi^2(\ell)$ denotes the chi-squared distribution with $\ell$ degrees of freedom.
\end{theorem}

Part \ch{b)} of Theorem~\ref{thm:bartlett} implies that each column of $Q$ is distributed uniformly over the unit sphere in $\mathbb F^n$.
Note that $x = z / \|z\|_2$, for a Gaussian random vector $z \in\mathbb F^n$, has the same distribution.
The following result provides the distribution of the entries of $x$; this result can be found for $\mathbb F=\mathbb R$ in~\cite{Dixon}.

\begin{lemma}\label{lem:unif_element_dist}
Consider a random vector $x$ distributed uniformly over the unit sphere in $\mathbb F^n$ for $n\ge 2$. Then the entries of $x$ are i.i.d.~with
\[
  |x_i|^2\sim \mathrm{Beta}\left(\frac{\phi}{2},\frac{\phi (n-1)}{2} \right), \quad i=1,\ldots,n,
\]
where $\mathrm{Beta}$ denotes the beta distribution.
\end{lemma}
\begin{proof}
By Theorem \ref{thm:bartlett}, the entries of $x$ are independent. Without loss of generality, let $i = 1$. Using that $x=z/\|z\|_2$
for a Gaussian random vector $z$ and setting $w = \left[\begin{matrix}z_2 & \ldots & z_n\end{matrix}\right]$, it follows that
$|x_1|^2=\frac{|z_1|^2}{|z_1|^2+\|w\|_2^2}$, where $z_1$, $w$ are independent and $\phi |z_1|^2\sim \chi^2(\phi)$,
$\phi \|w\|_2^2\sim \chi^2(\phi(n-1))$. This implies the claimed result; see, e.g., \cite[p. 320]{Eaton}.
\end{proof}

Our analysis will connect $\alpha$ and $\beta$ from \eqref{eq:alfa_beta} to the nullspaces of $k\times (k+1)$ standard Gaussian matrices,
which are characterized by the following result.
\begin{lemma} \label{lemma:prelim}
For a Gaussian random matrix $\Omega\in\mathbb F^{k\times (k+1)}$ with $k \ge 2$, let $x$ be a vector in the nullspace of
$\Omega$ such that $\|x\|_2 = 1$. Then, with probability one, $|x_i|$ is uniquely determined and satisfies
\begin{equation}\label{eq:uniform_beta}
  |x_i|^2 \sim \mathrm{Beta}\left(\frac{\phi}{2},\frac{\phi k}{2} \right).
\end{equation}
\end{lemma}
\begin{proof}
We assume that $\Omega$ has rank $k$, which holds with probability $1$. For a Gaussian random vector $\omega \in \mathbb F^{k+1}$
independent of \ch{$\Omega^*$}, consider the QR decomposition $[\ch{\Omega^*}, \omega] = QR$. Letting $x$ denote the last column of $Q$, it
follows that $x$ is orthogonal to the columns of $\ch{\Omega^*}$ or, in other words, $x$ is in the nullspace of $\Omega$. From
Theorem~\ref{thm:bartlett}, it follows that $x$ is distributed uniformly over the unit sphere in $\mathbb F^{k+1}$. The distribution~\eqref{eq:uniform_beta}
then follows from Lemma \ref{lem:unif_element_dist}. Finally, note that $|x_i|^2$ is uniquely determined because the nullspace of $\Omega$ has dimension $1$.
\end{proof}

\subsection{Statistics of $|\alpha|, |\beta|$}

The results above readily yield the distribution of $|\alpha|^2$ and $|\beta|^2$.
\begin{proposition}\label{prop:main}
For $\mathbb F \in \{\mathbb R, \mathbb C\}$, let $U,V$ be $n\times k$ independent random
matrices from the uniform distribution on the Stiefel manifold $\mathbb V_k^n(\mathbb F)$. Consider a \ch{finite}
simple eigenvalue $\lambda_i \in \mathbb F$ of a singular pencil $A-\lambda B$ with $A,B\in{\mathbb F}^{n\times n}$.
Let $x_i$ and $y_i$ be the right and left normalized eigenvectors of the perturbed regular pencil \eqref{pert} for
the eigenvalue $\lambda_i$ and let $\alpha, \beta$ be defined as in \eqref{eq:alfa_beta}.
Then $|\alpha|$ and $|\beta|$ are independent random variables and
\[
  |\alpha|^2, |\beta|^2 \sim \mathrm{Beta}\left(\frac{\phi}{2},\frac{\phi k}{2} \right),
\]
where $\phi=1$ for $\mathbb F = \RR$ and $\phi=2$ for $\mathbb F=\CC$.
\end{proposition}
\begin{proof} We will only prove the distribution for $\alpha$; the derivation for $\beta$ is entirely analogous.
By the unitary invariance of the uniform distribution \ch{over the Stiefel manifold} we may assume without
loss of generality that
$\ch{X_1} = [e_1\ \ldots\ e_k]$ and $x = e_{k+1}$\ch{, where $e_i$ is the $i$-th vector of the standard canonical basis}. By Theorem~\ref{thm:bartlett}, the matrix $V$ is obtained from the
QR factorization $\Omega = VR$ of an $n \times k$ Gaussian random matrix $\Omega$, with $R$ being invertible almost surely. We partition
\[
  V^* = \begin{bmatrix} V_1 & v_2 & \cdots \end{bmatrix} =
  R^{-*} \begin{bmatrix} \Omega_1 & \omega_2 & \cdots \end{bmatrix} =
  R^{-*} \Omega^*,
\]
such that $V_1,\Omega_1$ are $k\times k$ matrices and $v_2,\omega_2$ are vectors. Then
\[
   0 = V^* x_i = V^*X_1 a + V^*x \alpha =
  \begin{bmatrix} V_1 & v_2 \end{bmatrix}
  \begin{bmatrix} a \\ \alpha \end{bmatrix}
  = R^{-*} \begin{bmatrix} \Omega_1 & \omega_2 \end{bmatrix}
  \begin{bmatrix} a \\ \alpha \end{bmatrix}.
\]
Since submatrices of Gaussian random matrices are again Gaussian random matrices, this means that
$\begin{bmatrix} a \\ \alpha \end{bmatrix}$ is in the nullspace of a $k \times (k+1)$ Gaussian random matrix and has norm 1.
Thus, the result on the distribution of $\alpha$ follows from Lemma~\ref{lemma:prelim}.

The independence of $|\alpha|$ and $|\beta|$ follows from the independence of $U$ and
$V$ combined with the fact that $|\alpha|$ does not depend on $U$ and $|\beta|$ does not depend on $V$.
\end{proof}

\begin{remrk}\label{rem:real}
It is important to emphasize that the case $\mathbb F=\mathbb R$ in Proposition~\ref{prop:main}
not only requires $A,B,D_A,D_B$ to be real but also the eigenvalue \ch{$\lambda_i$} to be real.
\end{remrk}

\begin{remrk}An analysis similar to the one above was performed in~\cite[Proposition 6.5]{LotzNoferini}
and~\cite[Section 4.1]{KreGlib22} for unstructured perturbations. This analysis also starts from the
relation~\eqref{eq:alfa_beta} and then analyzes the distribution of $|\alpha|\cdot|\beta|$.
One significant difference in our case is that $\alpha$ and $\beta$ are independent due to the structure
of the perturbation in~\eqref{pert}, while this does not hold for the setting considered
in~\cite{LotzNoferini} and~\cite{KreGlib22}.
\end{remrk}

\subsection{Statistics of $|\alpha| \cdot |\beta|$}

As explained above, we aim at showing that the random variable $|\alpha||\beta|$ is unlikely to become tiny.
We start by computing the expected value of $|\alpha||\beta|$. Since $|\alpha|$ and $|\beta|$ are independent
random variables, we have $\mathbb E[|\alpha||\beta|]=\mathbb E[|\alpha|]\mathbb E[|\beta|]$.
The factors can be computed using the following result from~\cite[Lemma A.1]{LotzNoferini}.
\begin{lemma}\label{lem:beta_a_b}
Let $X\sim \mathrm{Beta}(a,b)$, where $a,b>0$. Then
\[
  \mathbb E\big[X^{1/2}\big]=\frac{\mathrm{B}(a+1/2,b)}{\mathrm{B}(a,b)},\quad \text{where}\quad
  \mathrm{B}(a,b)=\frac{\Gamma(a)\Gamma(b)}{\Gamma(a+b)}.
\]
\end{lemma}

To simplify the presentation, we will from now on denote the scalars $\alpha$ and $\beta$, in the setting of Proposition~\ref{prop:main}, as
$\alpha_{\mathbb F}$ and $\beta_{\mathbb F}$ with $\mathbb F \in \{\mathbb R, \mathbb C\}$.
Combining Proposition \ref{prop:main} and Lemma \ref{lem:beta_a_b} gives the following result.

\begin{lemma}\label{lem:expect}
Under the assumptions of Proposition~\ref{prop:main}, the following holds:
\begin{enumerate}
    \item[a)] $\displaystyle \mathbb E[|\alpha_\CC|]=\mathbb E[|\beta_\CC|]=\frac{\sqrt{\pi}\Gamma(k+1)}{2\Gamma(k+3/2)}$ \
    and \
    $\displaystyle\mathbb E\left[|\alpha_\CC||\beta_\CC|\right]=\frac{\pi \Gamma(k+1)^2}{4 \Gamma(k+3/2)^2}$.
    \item[b)] $\displaystyle \mathbb E[|\alpha_\RR|]=\mathbb E[|\beta_\RR|]=\frac{\Gamma((k+1)/2)}{\sqrt{\pi}\Gamma((k+2)/2)}$ \
    and \
    $\displaystyle \mathbb E\left[|\alpha_\RR||\beta_\RR|\right]=\frac{\Gamma((k+1)/2)^2}{\pi \Gamma((k+2)/2)^2}$.
\end{enumerate}
\end{lemma}

\noindent Table~\ref{tab:vartheta} contains the computed expected values for different $k$, using the results of Lemma~\ref{lem:expect}.

Using the well-known bounds
\begin{equation}\label{bound_gamma}
    \sqrt{x}\le \frac{\Gamma(x+1)}{\Gamma(x+1/2)}\le \sqrt{x+1/2},
\end{equation}
for $x>0$, the expected values of $|\alpha||\beta|$ from Lemma \ch{\ref{lem:expect}} can be bounded as
\begin{equation}\label{bound_alpha}
\frac{\pi}{4(k+1)}\le \mathbb E[|\alpha_\CC||\beta_\CC|] \le \frac{\pi}{4(k+1/2)}, \quad \frac{2}{\pi(k+1)}\le \mathbb E[|\alpha_\RR||\beta_\RR|] \le \frac{2}{\pi k}.
\end{equation}
\begin{table}[htb]
\centering
\caption{Expected values of $|\alpha||\beta|$ for $k=n-\nrank(A,B)=1,2,4,\ldots,64$.}\label{tab:vartheta}
{
\begin{tabular}{rrr}\hline \rule{0pt}{2.7ex}%
$k$ & $\mathbb E\left[|\alpha_\CC||\beta_\CC|\right]$  & $\mathbb E\left[|\alpha_\RR||\beta_\RR|\right]$ \\ \hline \rule{0pt}{2.3ex}%
$1$ & $0.44444$ & $0.40528$ \\
$2$ & $0.28444$ & $0.25000$ \\
$4$ & $0.16512$ & $0.14063$ \\
$8$ & $0.08972$ & $0.07477$ \\
$16$ & $0.04688$ & $0.03857$ \\
$32$ & $0.02398$ & $0.01959$ \\
$64$ & $0.01213$ & $0.00987$ \\ \hline
\end{tabular}
}
\end{table}%
One therefore expects that $\gamma_i$ underestimates the true values $\gamma(\lambda_i)$ by roughly a factor $1/k$.

For real matrices $A$ and $B$, one would prefer to use real matrices in the perturbation~\eqref{pert} as well, because
eigenvalue computations are performed more efficiently in real arithmetic. As we can see from Table \ref{tab:vartheta} as well as from the bounds~\eqref{bound_alpha}, the expected value of $|\alpha||\beta|$ for real perturbations is only slightly smaller than the one
for complex perturbations. However, as we will see in the following, the left tail of $|\alpha||\beta|$
is less favorable in the real case and it appears to be safer to use complex modifications of the original pencil even for real data.

\subsection{Bounds on left tail of $|\alpha| \cdot |\beta|$}

We will start with a simple tail bound that extends a result from~\cite[Proposition 4.3]{KreGlib22} to the complex case.

\begin{corollary}\label{cor:simple_bound}
Under the assumptions of Proposition~\ref{prop:main}, we have
\begin{equation}\label{eq:simple_boundC}
  \mathbb P(|\alpha_\CC| |\beta_\CC| < t) \le 2kt
\end{equation}
and
\begin{equation}\label{eq:simple_boundR}
  \mathbb P(|\alpha_\RR| |\beta_\RR| < t) \le \sqrt{8k t/\pi}
\end{equation}
for every $0\le t\le 1$.
\end{corollary}

\begin{proof}
 Using $\min\{|\alpha|^4,|\beta|^4\} \le |\alpha|^2 |\beta|^2$, we obtain
 \begin{align*}
  \mathbb P\left(|\alpha|^2 |\beta|^2 < t^2\right) &\le \mathbb P\left(\min\left\{|\alpha|^4,|\beta|^4\right\} < t^2\right) \le
  \mathbb P\left(|\alpha|^2 < t\right) + \mathbb P\left(|\beta|^2 < t\right) \\
 & = \frac{2}{\mathrm B(\phi/2,\phi k/2)} \int_0^t x^{\phi/2-1} (1-x)^{\phi k/2-1}\, \mathrm{d}x,
 \end{align*}
where $\phi=1$ for $\mathbb F = \RR$ and $\phi=2$ for $\mathbb F=\CC$.
 For $\phi = 2$, we obtain from $|1-x|\le 1$ that
 \[
  \mathbb P\left(|\alpha_\CC|^2 |\beta_\CC|^2 < t^2\right) \le \frac{2t}{\mathrm B(1, k)} = 2k t.
 \]
Similarly, for $\phi = 1$ we get
 \[
  \mathbb P\left(|\alpha_\RR|^2 |\beta_\RR|^2 < t^2\right) \le \frac{4\sqrt{t}}{\mathrm B(1/2, k/2)} \le \sqrt{\frac{8k t}{\pi}},
 \]
 where we used the bound \eqref{bound_gamma} to derive the last inequality.
 \end{proof}

 Although we know from Proposition \ref{prop:main} that $|\alpha|$ and $|\beta|$ are independent random
 variables, this is not used in Corollary \ref{cor:simple_bound}.
 The following proposition, where we exploit the fact that $|\alpha|$ and $|\beta|$ are independent, significantly improves the results of Corollary \ref{cor:simple_bound}.

\begin{proposition}\label{cor:better_bound}
Under the assumptions of Proposition~\ref{prop:main}, it holds that
\begin{equation}\label{eq:better_bound_CC}
  \mathbb P\left(|\alpha_\CC| |\beta_\CC| < t\right) \le k^2t^2(1-2\ln t)
\end{equation}
and
\begin{equation}\label{eq:better_bound_RR}
  \mathbb P\left(|\alpha_\RR| |\beta_\RR| < t\right) \le \frac{2k}{\pi} t(-\ln t)+{\cal O}(t), \quad t \to 0.
\end{equation}
\end{proposition}
\begin{proof}
To derive the bound~\eqref{eq:better_bound_CC} we first observe that, since $|\alpha_\CC|$ and $|\beta_\CC|$ are
independent by Proposition \ref{prop:main}, it holds that
\[
  \mathbb P(|\alpha_\CC| |\beta_\CC|\le t)=
  \mathbb P\left(|\alpha_\CC|^2 |\beta_\CC|^2\le t^2\right)=
  \iint_{\cal D}g(x;k)g(y;k)\,\mathrm{d}x\,\mathrm{d}y
\]
with ${\cal D}=\left\{(x,y)\in[0,1]\times[0,1]:\ xy\le t^2\right\}$ and
$
  g(x;k)=\frac{1}{\mathrm{B}(1,k)}(1-x)^{k-1}.
$
Using the bound $g(x;k)\le \frac{1}{\mathrm{B}(1,k)}=k$ gives~\eqref{eq:better_bound_CC}:
\[
  \mathbb P(|\alpha_\CC| |\beta_\CC|\le t)
  \le k^2 \iint_{\cal D}\mathrm{d}x\mathrm{d}y = k^2 t^2(1-2\ln{t}).
\]

For real perturbations we have $|\alpha_\RR|^2, |\beta_\RR|^2\sim \mathrm{Beta}(1/2,k/2)$ with the distribution function satisfying
\[
  h(x;k)=\frac{1}{\mathrm{B}(1/2,k/2)}x^{-1/2}(1-x)^{k/2-1} \le \frac{1}{\mathrm{B}(1/2,k/2)}x^{-1/2}.
\]
This gives
\begin{align*}
  \mathbb P(|\alpha_\RR| |\beta_\RR|\le t)&= \iint_{\cal D}h(x;k)h(y;k)\mathrm{d}x\mathrm{d}y\\
  &\le \frac{1}{B(1/2,k/2)^2} \iint_{\cal D}x^{-1/2}y^{-1/2}\mathrm{d}x\mathrm{d}y\\
  &=\frac {1}{B(1/2,k/2)^2} 4 t\left(2 \sqrt{\pi} - 1 + t -\ln{t}\right)\\
  &\le \frac{2 k}{\pi} t\left(2 \sqrt{\pi} - 1 + t -\ln{t}\right), 
\end{align*}
where we applied \eqref{bound_gamma}.
\end{proof}

Proposition~\ref{cor:better_bound} indicates that, even for a real singular pencil, it would be better to use  complex perturbations
as they give a much smaller probability of obtaining tiny $|\alpha||\beta|$. This is confirmed by the following lower bound for
$\mathbb P(|\alpha_\RR| |\beta_\RR|< t)$.

\begin{lemma}\label{lem:real_lower} Under the assumptions of Proposition~\ref{prop:main}, it holds that
\begin{equation}\label{eq:lower_bound_RR}
  \mathbb P\left(|\alpha_\RR| |\beta_\RR| < t\right) \ge \sqrt{\frac{8(k-1)}{\pi}} t+{\cal O }(t^2), \quad t\to 0.
\end{equation}
\end{lemma}
\begin{proof}
It is easy to see that
\begin{align*}
  \mathbb P(|\alpha_\RR| |\beta_\RR| < t)&\ge \mathbb  P(|\alpha_\RR|<t) +
  \mathbb  P(|\beta_\RR|<t) -\mathbb  P(|\alpha_\RR|<t)\mathbb  P(|\beta_\RR|<t)\\
  &=
  2\, \mathbb  P(|\alpha_\RR|<t)-\mathbb P(|\alpha_\RR|<t)^2.
\end{align*}
From the Taylor expansion of $\mathbb  P(|\alpha_\RR|<t)$ around $t=0$ we get
\[
  \mathbb P(|\alpha_\RR| |\beta_\RR| < t)\ge \frac{4}{\mathrm{B}(1/2,k/2)}t +{\cal O}(t^2)
\]
and the bound follows from applying \eqref{bound_gamma}.
\end{proof}

From Lemma \ref{lem:real_lower} and Proposition \ref{cor:better_bound} we see that for sufficiently small $t>0$ the
lower bound \eqref{eq:lower_bound_RR} for $\mathbb P\left(|\alpha_\RR| |\beta_\RR| < t\right)$
is much larger than the upper bound \eqref{eq:better_bound_CC} for $\mathbb P\left(|\alpha_\CC| |\beta_\CC| < t\right)$.

Figure \ref{fig:bounds} compares the obtained bounds to $\mathbb P\left(|\alpha_\CC| |\beta_\CC| < t\right)$ and
$\mathbb P\left(|\alpha_\RR| |\beta_\RR| < t\right)$, computed for $k=4$ and $k=8$ using the probability density functions from Appendix~\ref{sec:distr}.
The solid black line corresponds to $\mathbb P\left(|\alpha_\CC| |\beta_\CC| < t\right)$.
The blue dotted and dashed lines are the refined bound~\eqref{eq:better_bound_CC} from Proposition \ref{cor:better_bound} and the simple upper bound
\ch{\eqref{eq:simple_boundC}} from Corollary \ref{cor:simple_bound}, respectively. The solid red line corresponds to
$\mathbb P\left(|\alpha_\RR| |\beta_\RR| < t\right)$. The corresponding bounds are
magenta curves, which show the lower bound~\eqref{eq:lower_bound_RR} from Lemma~\ref{lem:real_lower},
the refined upper bound~\eqref{eq:better_bound_RR} from Proposition \ref{cor:better_bound},
and the simple upper bound~\ch{\eqref{eq:simple_boundR}} from Corollary \ref{cor:simple_bound}, respectively,
As expected, the bounds from Proposition \ref{cor:better_bound} are much
sharper than the simple bounds from Corollary \ref{cor:simple_bound}. Also, it is clearly seen from Figure~\ref{fig:bounds} that the probability of obtaining a tiny value for $|\alpha||\beta|$ is much larger when using real perturbations.

\begin{figure}[ht]
    \centering
    \includegraphics[width=5.8cm]{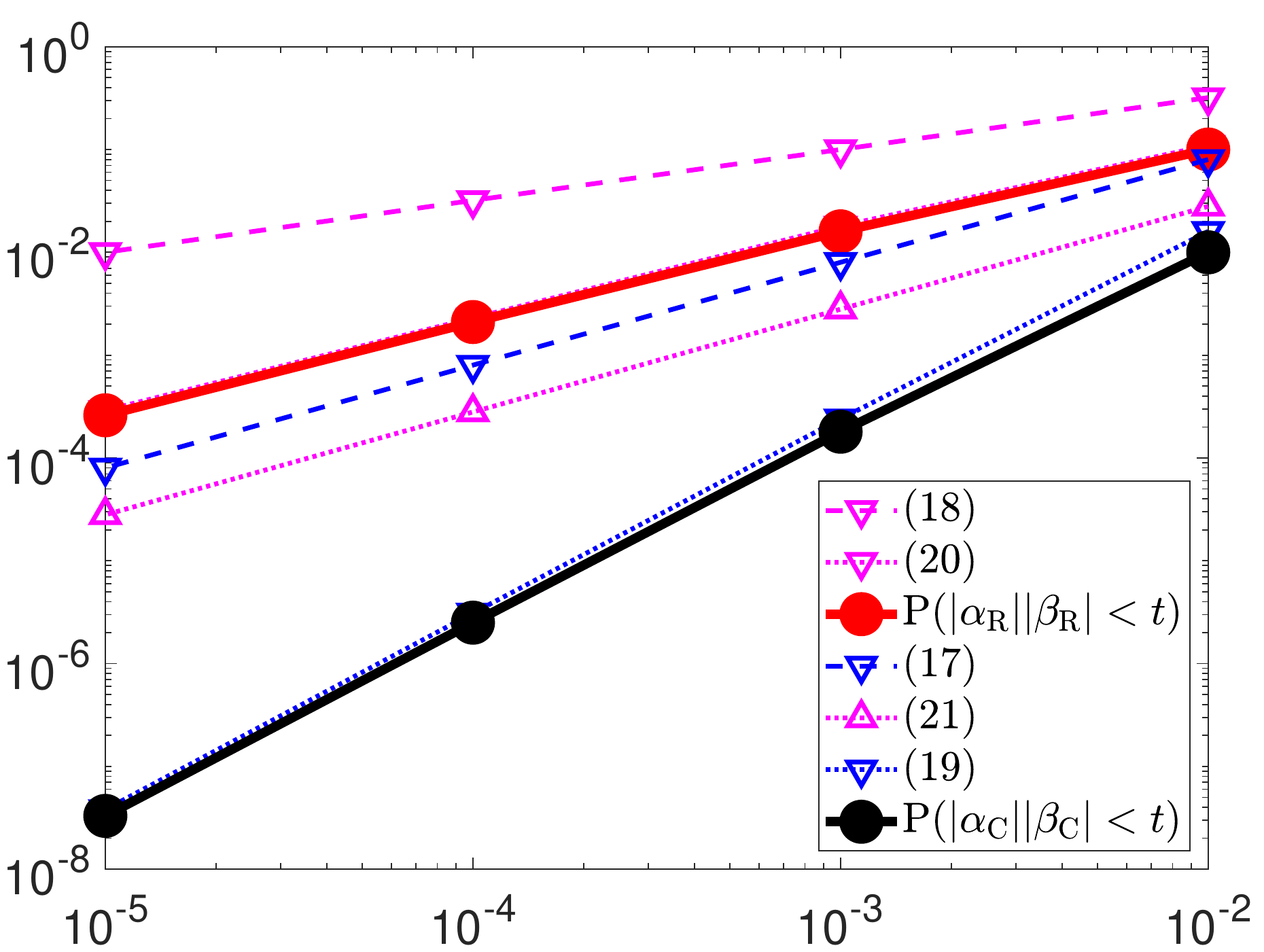}\
    \includegraphics[width=5.8cm]{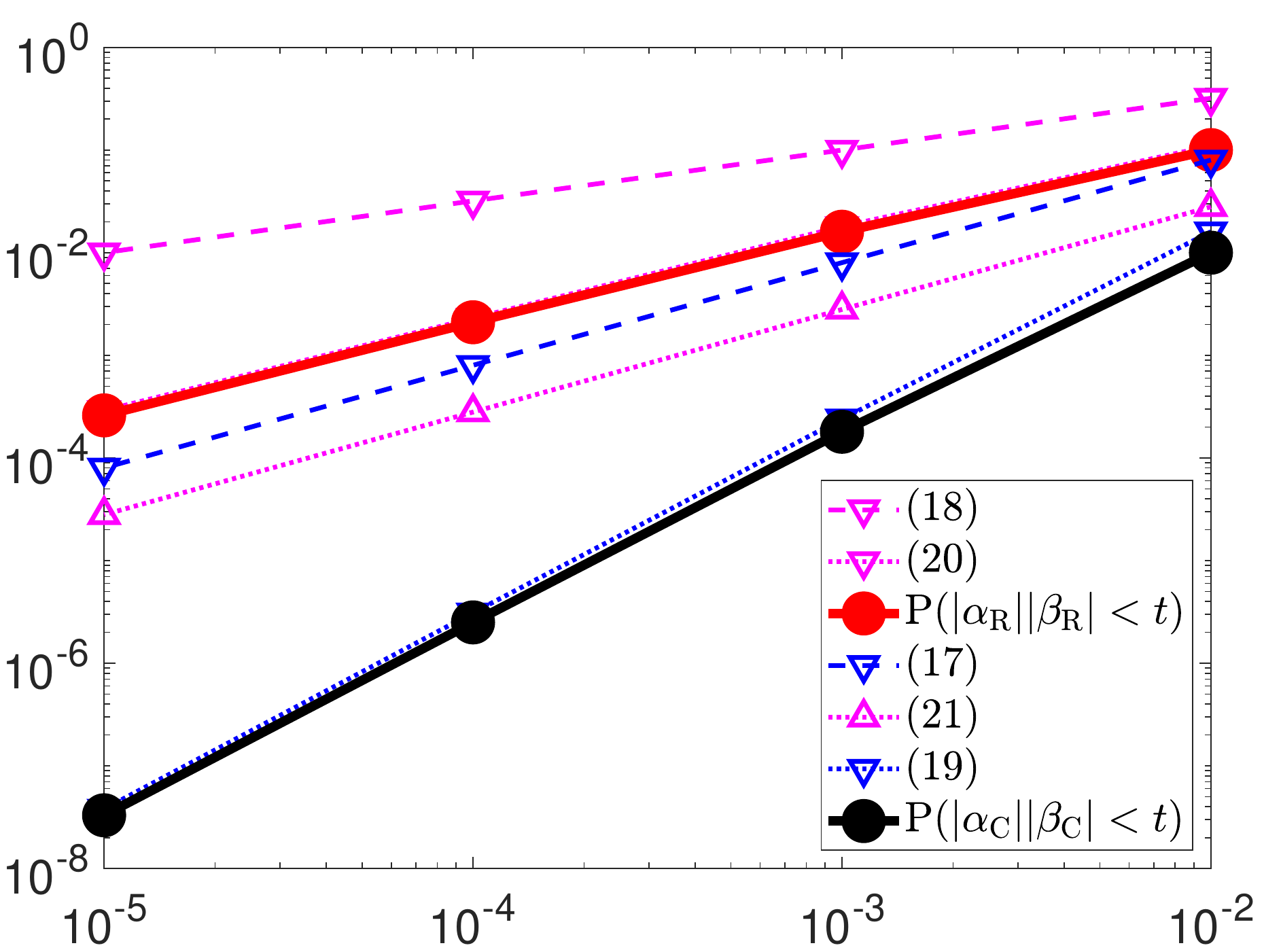}

    \caption{Comparison of bounds and actual values of $\mathbb P(|\alpha||\beta|<t)$ for
    $t=10^{-5},10^{-4},10^{-3},10^{-2}$, and for $k=4$ (left) and $k=8$ (right)}
    \label{fig:bounds}
\end{figure}

\begin{paragraph}{Summary}
 Together with the discussion in the beginning of this section, Proposition~\ref{cor:better_bound} allows us to compare the (reciprocal) eigenvalue sensitivity $\gamma(\lambda_i)$ of the original pencil with the corresponding quantity
 $\gamma_i$ for any of the three modified pencils used in Algorithms~1--3. For complex perturbations, we obtain that
 \[
    \gamma_i \le \gamma(\lambda_i) \le \gamma_i / t
\]
holds with probability at least $1-k^2t^2(1-2\ln t)$ for any $t > 0$.
\end{paragraph}

\section{Numerical examples}\label{sec:numresults}

All numerical examples were obtained with Matlab 2021b~\cite{Matlab}. We used the implementations of Algorithms~1--3
available as routine {\tt singgep} in MultiParEig~\cite{MultiParEig}.

\begin{example}\label{ex:ena}\rm
We consider the $8\times 8$ singular matrix pencil
\begin{equation}\label{eq:AB7x7} {\small
A={\left[\begin{array}{rrrrrrrr}
 -1 & -1 & -1 & -1 & -1 & -1 & -1 & 0\\
 1 & 0 & 0 & 0 & 0 & 0 & 0 & 0\\
 1 & 2 & 1 & 1 & 1 & 1 & 1 & 0\\
 1 & 2 & 3 & 3 & 3 & 3 & 3 & 0\\
 1 & 2 & 3 & 2 & 2 & 2 & 2 & 0\\
 1 & 2 & 3 & 4 & 3 & 3 & 3 & -1\\
 1 & 2 & 3 & 4 & 5 & 5 & 4 & 1\\
 0 & 0 & 0 & 0 & 2 & 2 & 1 & 2\end{array}\right]},\quad
B={\left[\begin{array}{rrrrrrrr}
 -2 & -2 & -2 & -2 & -2 & -2 & -2 & 0 \\
 2 & -1 & -1 & -1 & -1 & -1 & -1 & 0 \\
 2 & 5 & 5 & 5 & 5 & 5 & 5  & 0\\
 2 & 5 & 5 & 4 & 4 & 4 & 4  & 0\\
 2 & 5 & 5 & 6 & 5 & 5 & 5  & -1\\
 2 & 5 & 5 & 6 & 7 & 7 & 7  & 1\\
 2 & 5 & 5 & 6 & 7 & 6 & 6 & 1\\
 0 & 0 & 0 & 0 & 0 & -1 & -1 & 0\end{array}\right]}},
\end{equation}
which is constructed so that the KCF contains blocks of all four possible types.
It holds that ${\rm nrank}(A,B)=6$, and the KCF has blocks $J_1(1/2)$, $J_1(1/3)$, $N_1$, $L_0$, $L_1$, $L_0^T$, and $L_2^T$.

Algorithm~1 was applied $10^5$ times using random real and complex modifications; we compared the
computed values of $\gamma_1$ to the exact value of $\gamma(\lambda_1)$
for the eigenvalue $\lambda_1=1/3$. The histograms of $\gamma_1/\gamma(\lambda_1)$
for real and complex modifications together with the corresponding
\ch{probability density function (pdf)} from Appendix~\ref{sec:distr} are
presented in Figure \ref{fig:hist4}. The histograms appear to be consistent with the pdfs. The computed average values of
$|\alpha||\beta|$ are $0.28437$ for complex and $0.24934$ for real modifications,
which are both close to the theoretically predicted values for $k=2$ in Table~\ref{lem:beta_a_b}.
We note that we get almost identical results for the other eigenvalue $\lambda_2=1/2$.

\begin{figure}[ht]
    \centering
    \includegraphics[width=5.8cm]{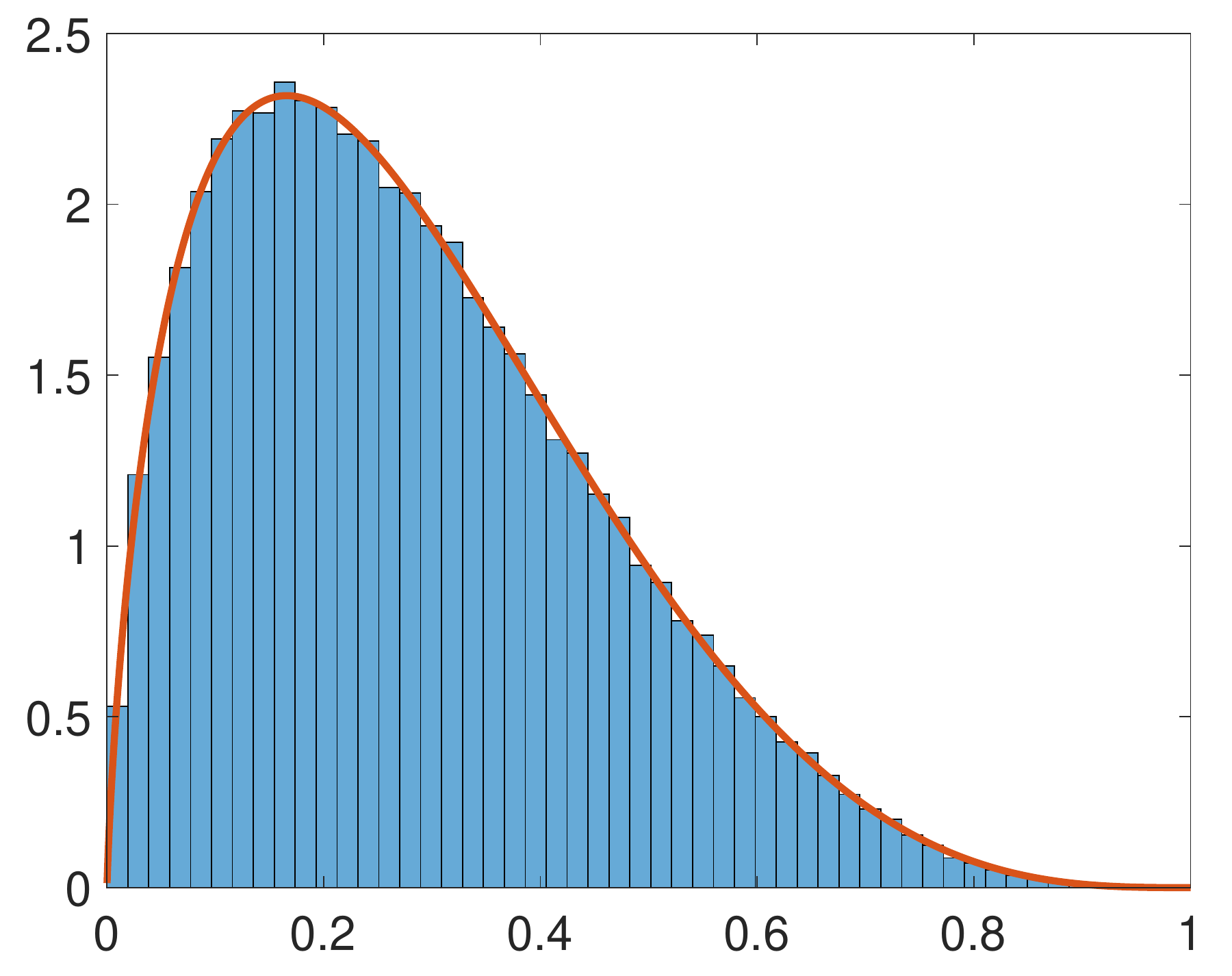}\
    \includegraphics[width=5.8cm]{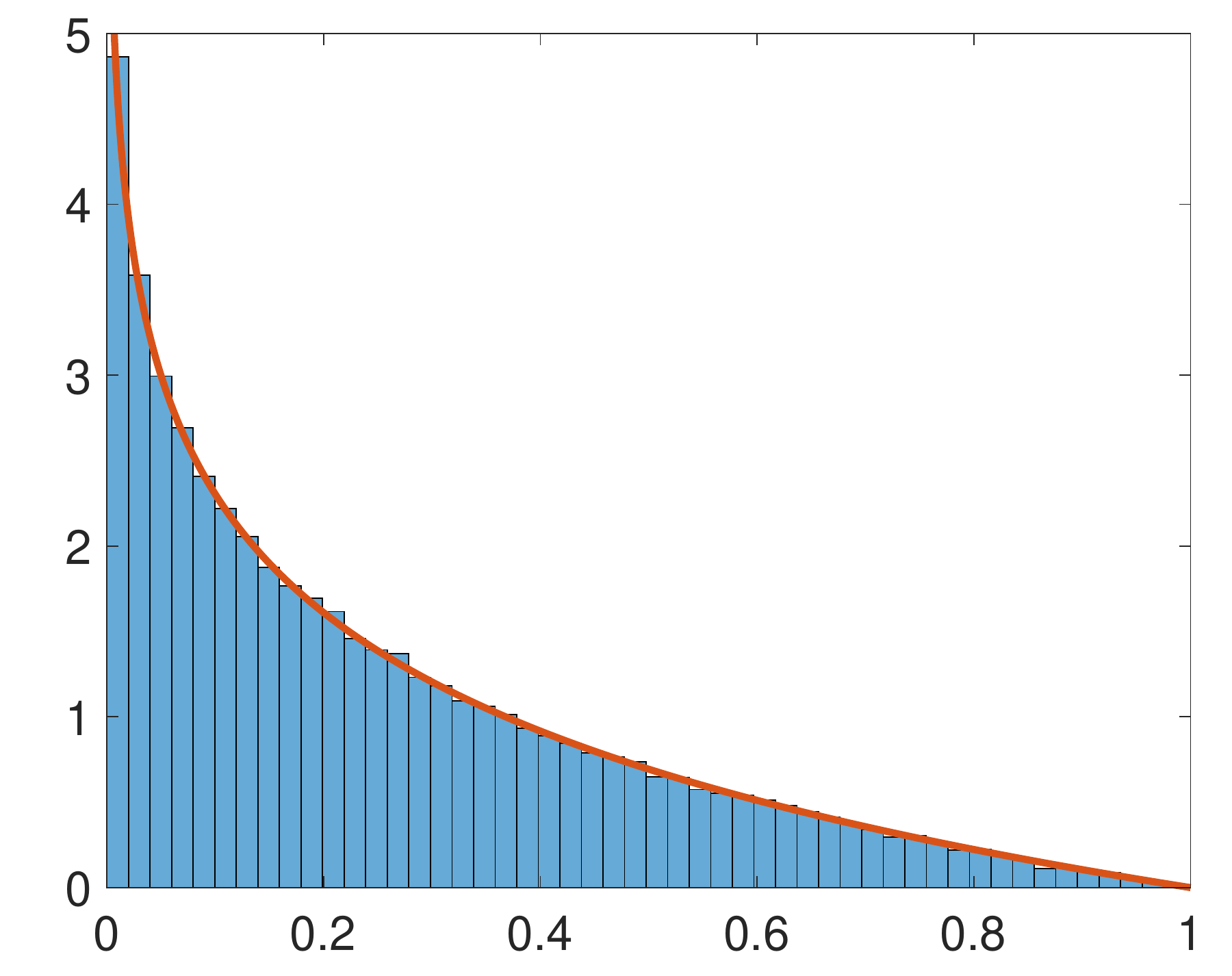}

    \caption{Example~\ref{ex:ena}: Histogram of $\gamma_1/\gamma(\lambda_1)$ and pdf using complex (left) and real (right) modifications}
    \label{fig:hist4}
\end{figure}
\end{example}

\begin{example}\label{ex:dva}\rm For the second example we consider
the pencil $\Delta_1-\lambda \Delta_0$ from \cite[Ex.~7.1]{HMP19} with matrices
\[\Delta_0=B_1\otimes C_2-C_1\otimes B_2,\quad
  \Delta_1=C_1\otimes A_2-A_1\otimes C_2
\]
of size $25\times 25$
related to the two-parameter eigenvalue problem (for details see, e.g., \cite[Sec. 7]{HMP19}) of the form
\begin{align*}
A_1+\lambda B_1+\mu C_1 &=\small{\left[\begin{array}{ccccc}
   0 & 0 & 4 + 7\lambda & 1 & 0\\
   0 & 5 + 8\lambda & 2 & -\lambda & 1 \\
   6 + 9\lambda + 10\mu & 3 & 1 & 0 & -\lambda\\
   1 & -\mu & 0 & 0 & 0\\
   0 & 1 & -\mu & 0 & 0 \end{array}\right]}, \\
A_2+\lambda B_2+\mu C_2 &=\small{\left[\begin{array}{ccccc}
   0 & 0 & 7 + 4\lambda & 1 & 0\\
   0 & 6 + 3\lambda & 9 & -\lambda & 1 \\
   5 + 2\lambda + \mu & 8 & 10 & 0 & -\lambda\\
   1 & -\mu & 0 & 0 & 0\\
   0 & 1 & -\mu & 0 & 0 \end{array}\right],}
\end{align*}
where
{\small \begin{align*}
\det(A_1+\lambda B_1+\mu C_1) &= 1 + 2\lambda + 3\lambda + 4\lambda^2 +
5\lambda \mu + 6\mu^2 + 7\lambda^3 +
8\lambda^2\mu +
9\lambda\mu^2 + 10\mu^3,\\
\det(A_2+\lambda B_2+\mu C_2) &= 10 + 9\lambda + 8\mu + 7\lambda^2 + 6\lambda\mu +
 5\mu^2 + 4\lambda^3 + 3\lambda^2\mu + 2\lambda\mu^2 + \mu^3.
\end{align*}}
The normal rank of $\Delta_1-\lambda \Delta_0$ is $21$ and the KCF
contains 4 $L_0$, 4 $L_0^T$, 2 $N_4$, 1 $N_2$, 2 $N_1$, and 9 $J_1$ blocks. Its finite
eigenvalues are $\lambda$-components of the 9 solutions of the system of two
bivariate polynomials $\det(A_1+\lambda B_1+\mu C_1)=0$ and
$\det(A_2+\lambda B_2+\mu C_2)=0$.
The pencil $\Delta_1-\lambda \Delta_0$ has one real eigenvalue $\lambda_1=-2.41828$ and
eight complex eigenvalues $\lambda_2,\ldots,\lambda_9$.

Algorithm~2 was applied $10^5$ times using random real and complex projections. We compared the computed values of
$\gamma_1$ to the exact value of $\gamma(\lambda_1)$
for the real eigenvalue $\lambda_1$. In both cases the method successfully
computed all finite eigenvalues. The histograms of $\gamma_1/\gamma(\lambda_1)$
together with the corresponding pdf from Appendix~\ref{sec:distr} are presented in Figure
\ref{fig:hist8}. Similar to the previous example, both histograms appear to be consistent with the pdfs. The computed average values of
$|\alpha||\beta|$ are $0.16413$ for complex and $0.14124$ for real projections;
both are again very close to the values in Table~\ref{lem:beta_a_b} for $k=4$.

\begin{figure}[ht]
    \centering
    \includegraphics[width=5.8cm]{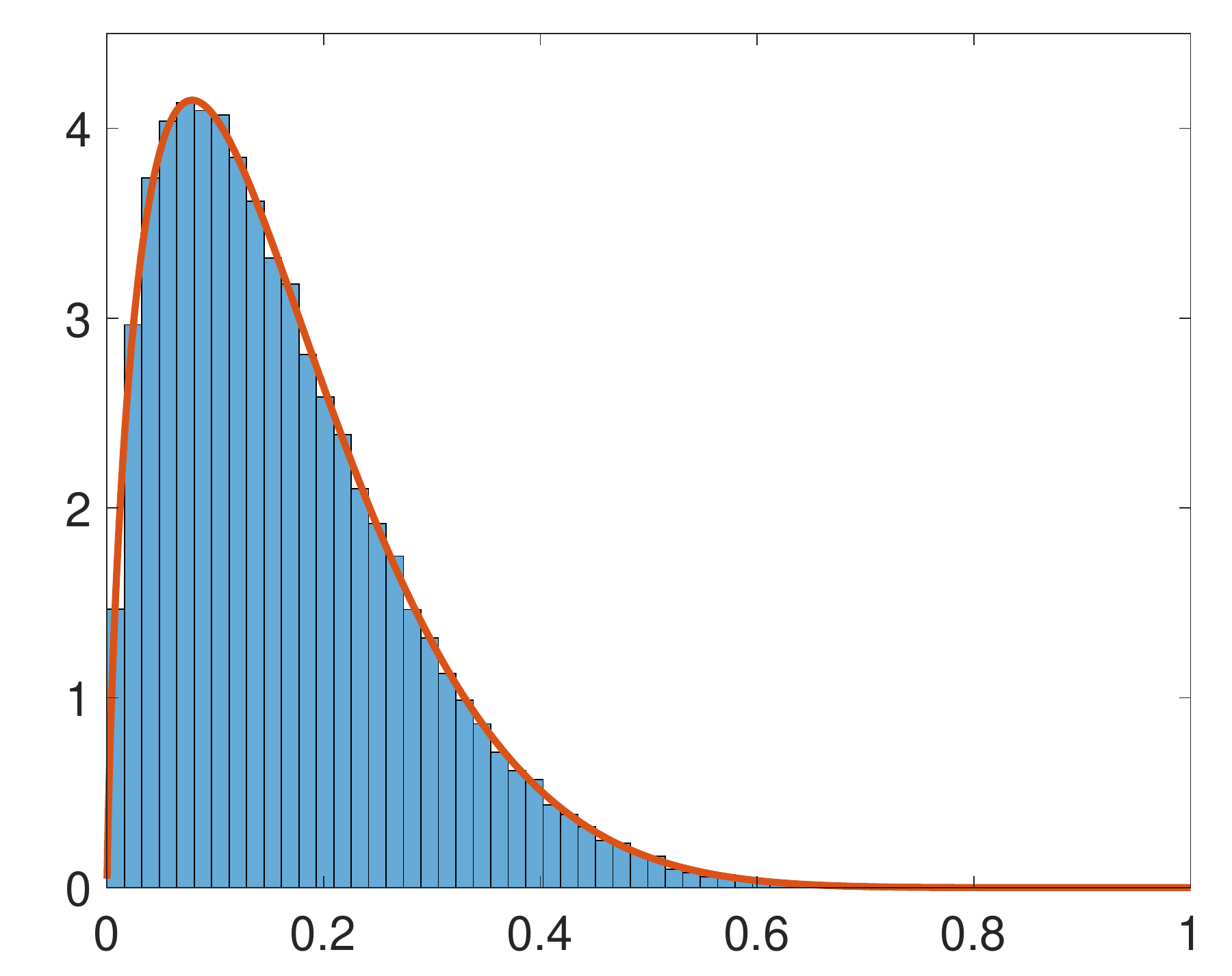}\
    \includegraphics[width=5.8cm]{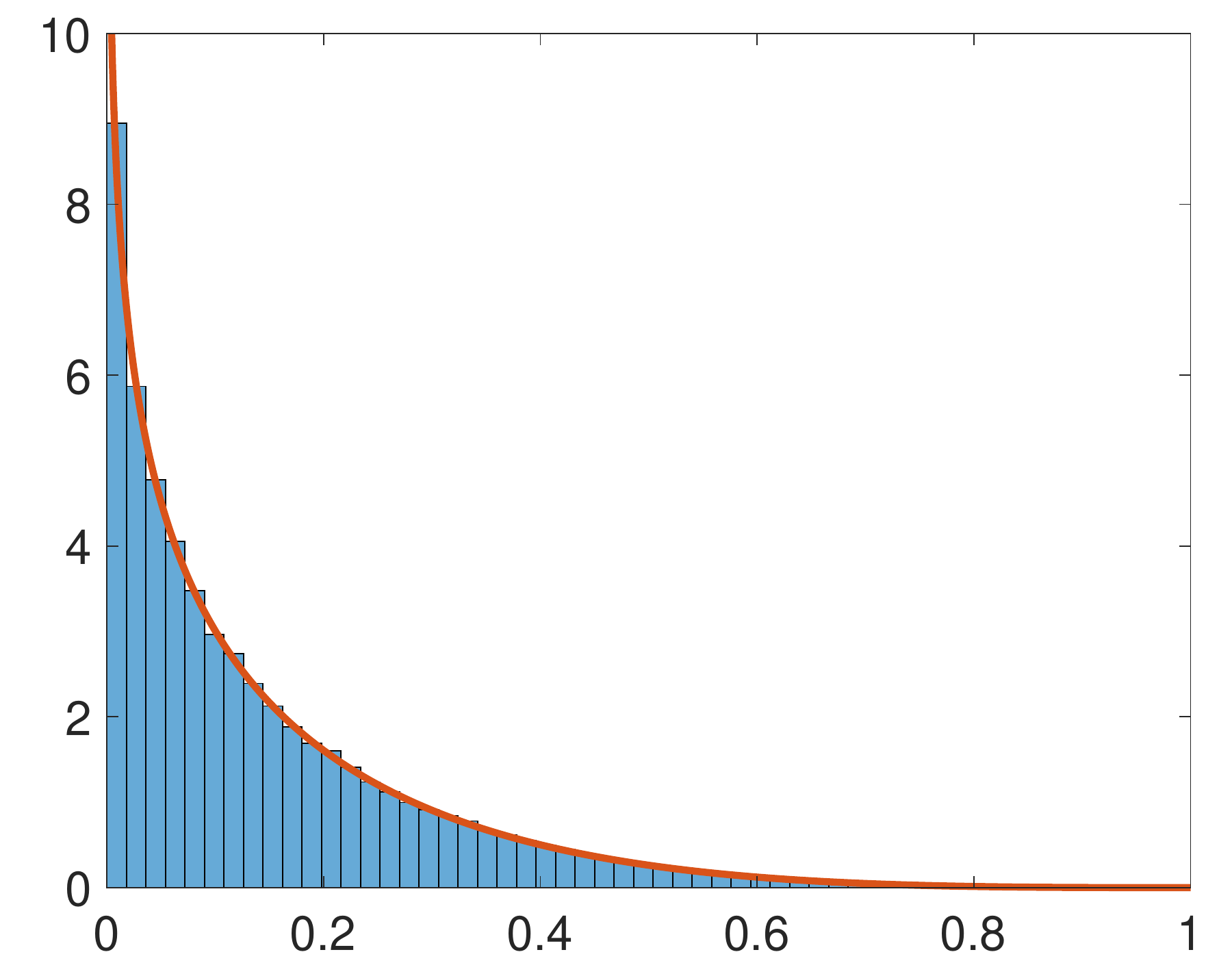}

    \caption{Example~\ref{ex:dva}: Histogram of $\gamma_1/\gamma(\lambda_1)$ and pdf using complex (left) and real (right) projections}
    \label{fig:hist8}
\end{figure}
\end{example}

Let us note that we get essentially identical results if we exchange the methods and use Algorithms~2 or 3 in Example~\ref{ex:ena}
and Algorithms~1 or 3 in Example~\ref{ex:dva}, as expected by Corollary \ref{cor:zveza}.
However, if we apply Algorithm~1 or 3 with real modifications or Algorithm~2 with real projections
to the real pencil $\Delta_1-\lambda \Delta_0$,
and consider any of the  8 complex eigenvalues $\lambda_2,\ldots,\lambda_9$,
we get results that cannot be explained by Proposition~\ref{prop:main} and behave
like the results in the next example.

\begin{example}\label{ex:tri}\rm
In this example we study the effect of real projections on a complex eigenvalue of a real pencil,
which is a situation that is not covered by Proposition~\ref{prop:main}. By considering two
equivalent singular pencils we will show that the
distribution function for $|\alpha_\RR||\beta_\RR|$, when we apply real projections to complex eigenvalues of real
pencils, depends on more than just the difference $k$ between the size of the pencil and its normal rank, which is the key value in Proposition~\ref{prop:main}.

We take block diagonal matrices
\begin{equation}\label{eq:AB7x7new}
A_0={\left[\begin{array}{rr}
 e_1e_5^T &  \\
  & A_{20} \end{array}\right]},\quad
B_0={\left[\begin{array}{rr}
 e_1e_4^T &  \\
  & B_{20} \end{array}\right]},
\end{equation}
where $e_1,e_4,e_5$ are standard basis vectors in $\RR^5$, and
\[A_{20}=\footnotesize\left[\begin{array}{rrrrr}
 1 & 0 & 0 & 0 & 0 \\
 0 & 2 & 0 & 0 & 0  \\
 0 & 0 & 1 & 1 & 0 \\
 0 & 0 & -1 & 1 & 0  \\
 0 & 0 & 0 & 0 & 1 \end{array}\right],\quad
B_{20}=\footnotesize\left[\begin{array}{rrrrr}
 1 & 0 & 0 & 0 & 0 \\
 0 & 1 & 0 & 0 & 0  \\
 0 & 0 & 1 & 0 & 0 \\
 0 & 0 & 0 & 1 & 0  \\
 0 & 0 & 0 & 0 & 0 \end{array}\right].
 \]
The $10\times 10$ pencil $A_0-\lambda B_0$ has
${\rm nrank}(A_0,B_0)=6$ and the KCF has blocks $J_1(1+\mathrm{i})$, $J_1(1-\mathrm{i})$, $J_1(2)$, $N_1$, 3 $L_0$, $L_1$, 3 $L_0^T$, and $L_1^T$.
We multiply $A_0$ and $B_0$ into $A_i=Q_iA_0Z_i$ and $B_i=Q_iB_0Z_i$ by \ch{real} matrices $Q_i$ and $Z_i$ \ch{whose entries are
independent random variables uniformly distributed on $(0,1)$} for $i=1,2$
to get two equivalent singular pencils of the same size, normal rank and eigenvalues.

Algorithm~2 was applied $10^5$ times using random real projections to each of the pencils $A_1-\lambda B_1$ and $A_2-\lambda B_2$
and the computed value of $\gamma_1$ was compared to the exact value of $\gamma(\lambda_1)$ for the complex eigenvalue $\lambda_1=1+\mathrm{i}$.
The histograms of $\gamma_1/\gamma(\lambda_1)$ together with the theoretical distribution functions from Section~\ref{sec:distr} are
presented in Figure \ch{\ref{fig:histRC}}.
We see that, although the pencils $A_1-\lambda B_1$ and $A_2-\lambda B_2$ are equivalent, the  histograms are different. The histograms also
look different than in the case when we use real projections for a real eigenvalue for $k=4$
 (Figure~\ref{fig:hist4} \ch{right}) or complex projections (Figure~\ref{fig:hist4} \ch{left}).
  While the shape of the left histogram in Figure~\ref{fig:histRC} resembles the
 shape expected for complex perturbations, is the shape of the right histogram more in line with the
 distribution function for real perturbations.
The computed average values of $|\alpha_\RR||\beta_\RR|$ are also
different, we get $0.12195$ for $A_1-\lambda B_1$ and $0.13623$ for $A_2-\lambda B_2$, both values are completely
different from the values in Table~\ref{lem:beta_a_b} for $k=4$.

Since we get histograms of different shape for two real equivalent pencils, this shows that the
distribution function for $|\alpha_\RR||\beta_\RR|$, when we apply real perturbations to complex eigenvalues of real
pencils, depends on more than just the structure of the KCF.

 \begin{figure}[ht]
    \centering
    \includegraphics[width=5.8cm]{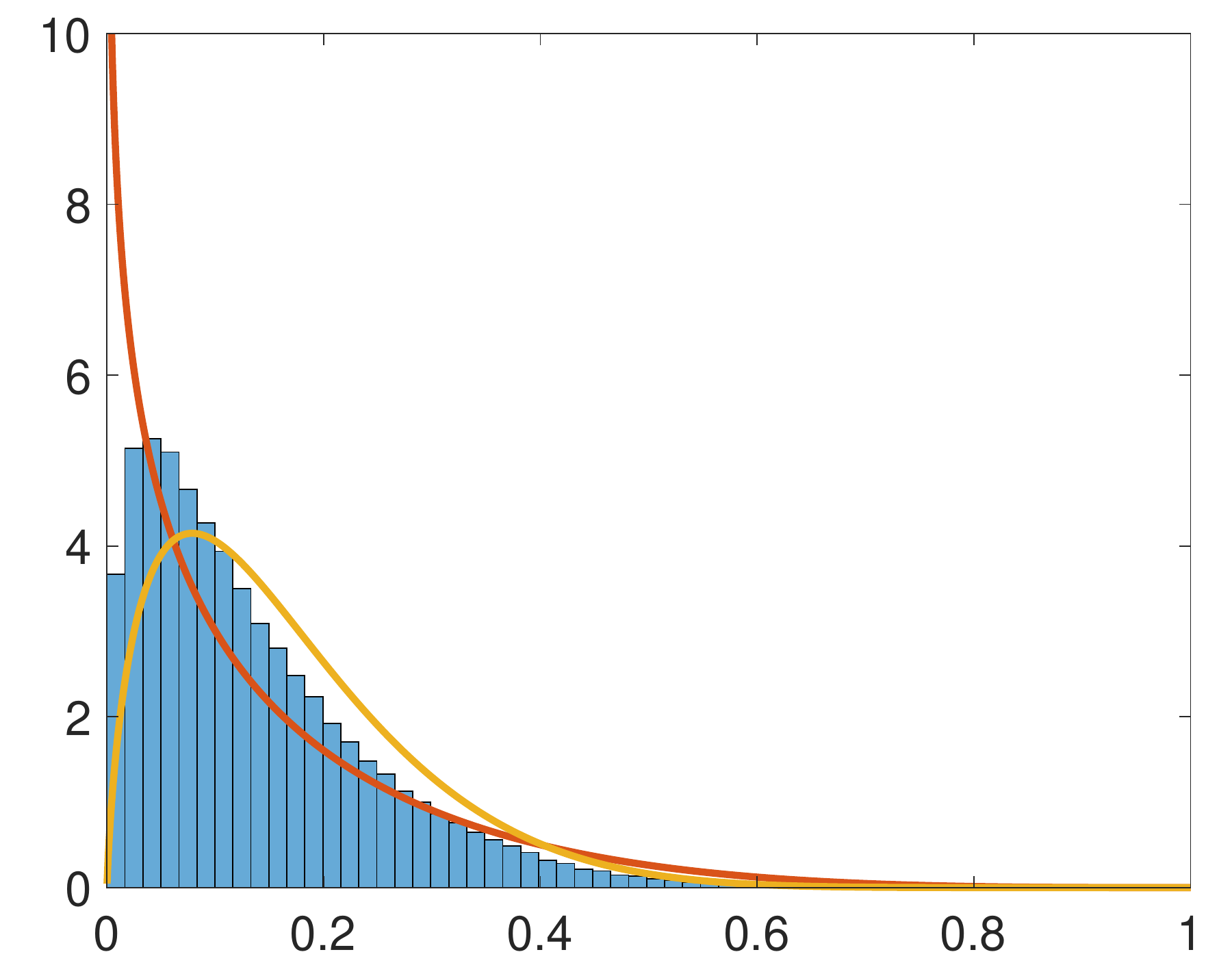}\
    \includegraphics[width=5.8cm]{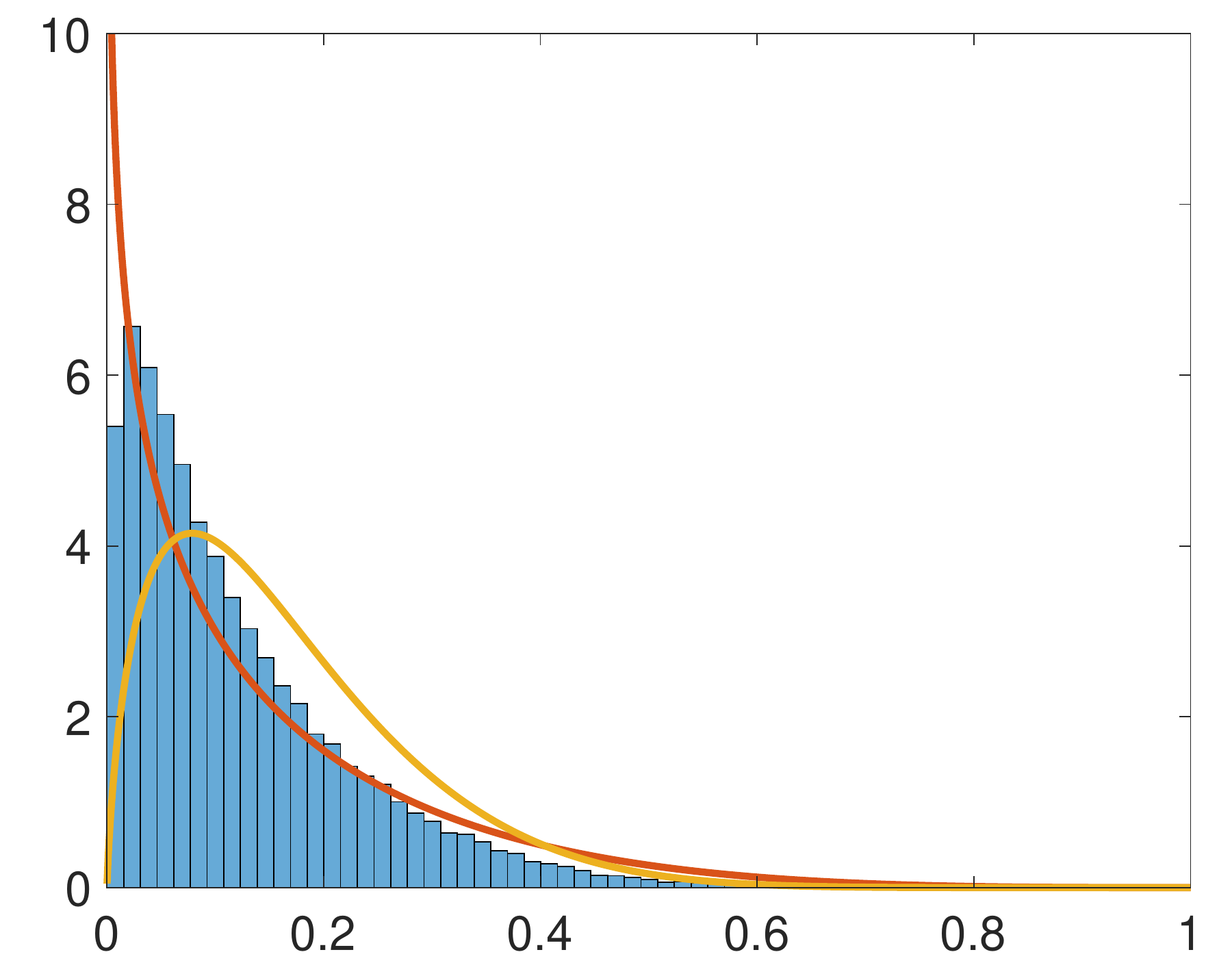}

    \caption{Example~\ref{ex:tri}: Histogram of $\gamma_1/\gamma(\lambda_1)$ and theoretical distribution functions
    (real and complex) for $k=4$ using real perturbations for a complex eigenvalue of a real
    pencil for $A_1-\lambda B_1$ (left) and $A_2-\lambda B_2$ (right).}
    \label{fig:histRC}
\end{figure}

We remark that the histograms (not shown) for the real eigenvalue $\lambda_3=2$ for
both pencils look identical to the right picture in Figure~\ref{fig:hist8} from
Example~\ref{ex:dva}, where $k=4$ as well, which agrees with Remark~\ref{rem:real_complex} that even if some
eigenvalues of the real pencil are complex, this does not affect the behavior of real perturbations
to real eigenvalues. The computed average values of $|\alpha_\RR||\beta_\RR|$ for the real eigenvalue $\lambda_3$ are
$0.14077$ for $A_1-\lambda B_1$ and $0.14110$ for $A_2-\lambda B_2$, they both
agree with the value in Table~\ref{lem:beta_a_b} for $k=4$.
\end{example}
\smallskip

The last numerical example reflects that the case of real perturbations
for a complex eigenvalue of a real singular pencil is not covered by
Proposition~\ref{prop:main}. Indeed, this case was not taken properly into account
in~\cite{LotzNoferini} and it remains an open problem to derive an expression or a tight simple bound for the
$\delta$-weak condition number of a complex eigenvalue under real perturbations.

\section{Conclusions}\label{sec:conclus}
We have analyzed three random based numerical methods for computing finite eigenvalues of a singular matrix pencil.
All algorithms are based on random matrices that transform the
original singular pencil into a regular one in such way that the eigenvalues remain intact. Our analysis confirms
the numerical validity of these methods with high probability.

We also obtained sharp \ch{left tail bounds} on the distribution of a product of two independent random variables
distributed with the generalized beta distribution of the first kind or Kumaraswamy distribution.

\begin{paragraph}{\bf Acknowledgments} We thank Ivana \v{S}ain Glibi\'{c} for discussions related to this work.
\ch{The authors also thank the reviewers for their careful reading and useful comments.}
\
 \end{paragraph}

\section*{Declarations}

\begin{paragraph}{\bf Conflict of interests statement} Not applicable.
\end{paragraph}

\bibliography{references}
\bibliographystyle{abbrv}

\appendix
\section{A closer look at the distribution of $|\alpha||\beta|$}\label{sec:distr}
In this section, we derive explicit expressions for the probability density function (pdf) of $|\alpha||\beta|$.
For this purpose, we first write down the pdfs of $|\alpha|$, $|\beta|$\ch{, which we can
express with the generalized beta distribution of the first kind \cite{McDonald} in the real case and the Kumaraswamy distribution \cite{Jones} in the complex case}.

\begin{lemma} \label{lemma:dist} Under the assumptions of Proposition~\ref{prop:main}, we have
\begin{enumerate}
    \item[a)] $|\alpha_\CC|, |\beta_\CC|\sim \mathrm{GB}_1(2,1,1,k)=\mathrm{Kumaraswamy}(2,k)$ with the pdf
\[
  h(x;k)=2k x(1-x^2)^{k-1},
\]
\item[b)] $|\alpha_\RR|, |\beta_\RR| \sim \mathrm{GB}_1(2,1,\frac{1}{2},\frac{k}{2})$ with the pdf
\[
  g(x;k)=\frac{2}{B(\frac{1}{2},\frac{k}{2})}(1-x^2)^{k/2-1},
\]
\end{enumerate}
where $\mathrm{GB}_1(a,b,p,q)$ is the generalized beta distribution of the first kind and
$\mathrm{Kumaraswamy}(p,q)$ denotes the Kumaraswamy distribution with parameters $p$ and $q$.
\end{lemma}
\begin{proof}
The pdf of $X\sim \mathrm{Beta}(p,q)$ is given by
\[f(x;p,q)=\frac{1}{\mathrm{B}(p,q)}x^{p-1}(1-x)^{q-1}.\]
It follows that the pdf for $Y = X^{1/2}$ is
\[2yf(y^2;p,q)=\frac{2}{\mathrm{B}(p,q)}y^{2p-1}(1-y^2)^{q-1}.\]
Now, we apply Proposition \ref{prop:main} and insert the corresponding values of $p$, $q$. The proof is concluded
by identifying the obtained densities with the Kumaraswamy distribution in the complex case and the generalized beta distribution of the
first kind in the real case; see, e.g., \cite{Jones,McDonald}.
\end{proof}

The references~\cite{Jones,McDonald} used in the proof above also provide expressions for the moments of the
two distributions in Lemma~\ref{lemma:dist}, which can be used to verify the results of Lemma~\ref{lem:expect}.

In order to derive the pdf of $|\alpha|\cdot |\beta|$  from the result of Lemma~\ref{lemma:dist}, we have performed
symbolic computations in Wolfram Mathematica~\cite{Mathematica}.
Let $f_\CC(x;k)$ and $f_\RR(x;k)$  denote the pdfs for $|\alpha_\CC| |\beta_\CC|$ and $|\alpha_\RR| |\beta_\RR|$ , respectively.
Using that $|\alpha_\CC|, |\beta_\CC|$ are independent and
$|\alpha_\CC|, |\beta_\CC|\sim \mathrm{Kumaraswamy}(2,k)$, we obtain that
\begin{align*}
 f_\CC(x;k)&=
2k^2 x(x^2-1)^{2k-1}\mathrm{B}(k,k+1)\Big(\frac{k(x^2-1)}{2k+2}\leftindex_2F_1(k+1,k+1,2k+2,1-x^2)\\
&-2\leftindex_2F_1(k,k,2k+1,1-x^2)\Big),
\end{align*}
where $\leftindex_2F_1(a,b,c,z)$ is the hypergeometric function. One finds that $f_\CC$ takes the form
\begin{equation}\label{eq:fCC}
    f_\CC(x;k) = x\big( p_{\ch{k-1}}(x^2)+ q_{\ch{k-1}}(x^2)\ln x\big),
\end{equation}
where $p_{\ch{k-1}}$ and $q_{\ch{k-1}}$ are polynomials of degree $k-1$. Some explicit expressions for small $k$ are
\begin{align*}
    f_\CC(x;1) &= -4x \ln x,\\
    f_\CC(x;2) &= 16x \big(-1+x^2 - (1+x^2)\ln x\big),\\
    f_\CC(x;3) &= 18x \big(-3+3x^4 - 2(1+4x^2+x^4)\ln x\big),\\
    f_\CC(x;4) &= \frac{32}{3} x\big(-11-27x^2+27x^4+11x^6-6(1+9x^2+9x^4+x^6)\ln x\big).\\
\end{align*}

We were unable to obtain a closed form for the distribution of $|\alpha_\RR| |\beta_\RR|$, with
$|\alpha_\RR|, |\beta_\RR|$ independent and
 $|\alpha_\RR|, |\beta_\RR| \sim \mathrm{GB}_1(2,1,\frac{1}{2},\frac{k}{2})$.
For $k=2m$ we conjecture that
\[
f_\RR(x;2m) = p_{\ch{m-1}}(x^2)+ q_{\ch{m-1}}(x^2)\ln x,
\]
where $p_{\ch{m-1}}$ and $q_{\ch{m-1}}$ are polynomials of degree $m-1$. In addition to the fact that this looks similar to
\eqref{eq:fCC}, this conjecture is also supported by the following expressions for small $m$:
\begin{align*}
    f_\RR(x;2) &= -\ln x,\\
    f_\RR(x;4) &= \frac{9}{4}\big(-1+x^2-(1+x^2)\ln x\big),\\
    f_\RR(x;6) &= \frac{225}{128} \big(-3+3x^4-2(1+4x^2+x^4)\ln x\big).
\end{align*}

The graphs of $f_\CC(x;k)$ and $f_\RR(x;k)$ for $k=2,4,6,8$ are presented in Figure \ref{fig:pdf_c}. Note that
$\lim_{x\to 0}f_\CC(x;k)=0$ and $\lim_{x\to 0}f_\RR(x;k)=\infty$.
\begin{figure}[ht]
    \centering
    \includegraphics[width=5.8cm]{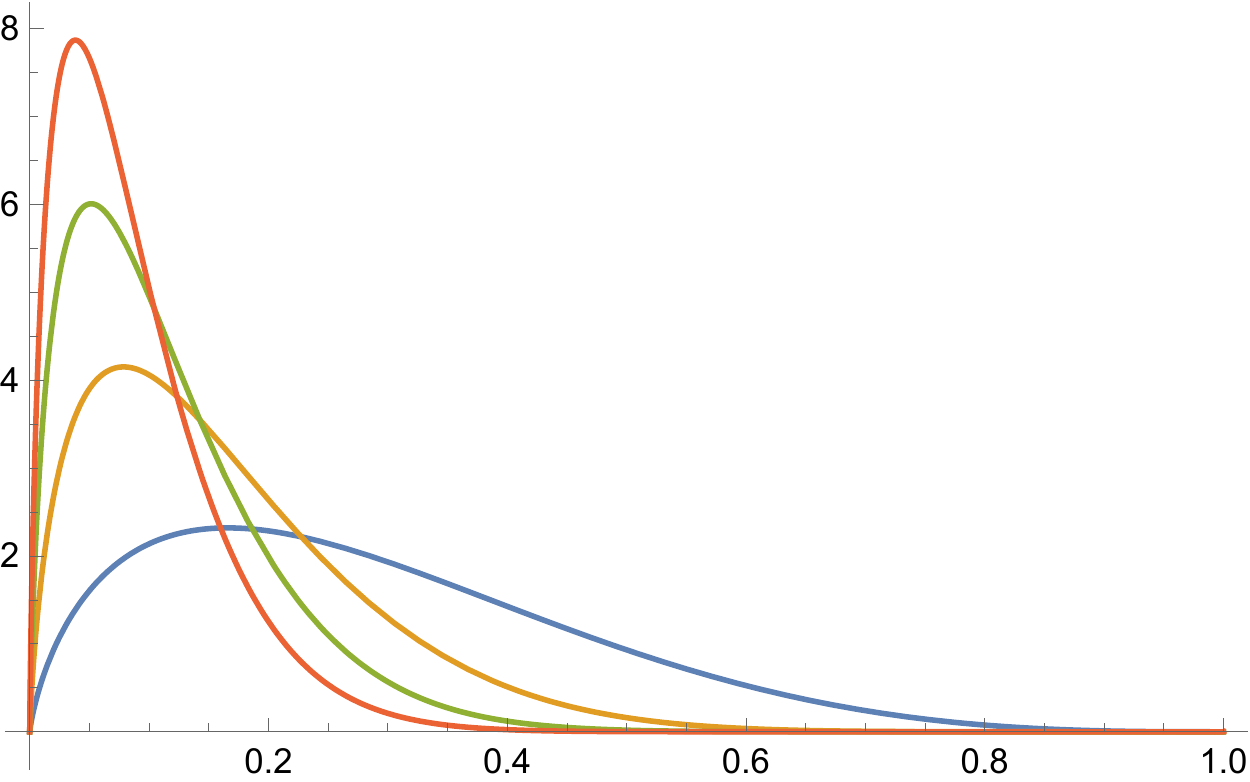}\ \
    \includegraphics[width=5.8cm]{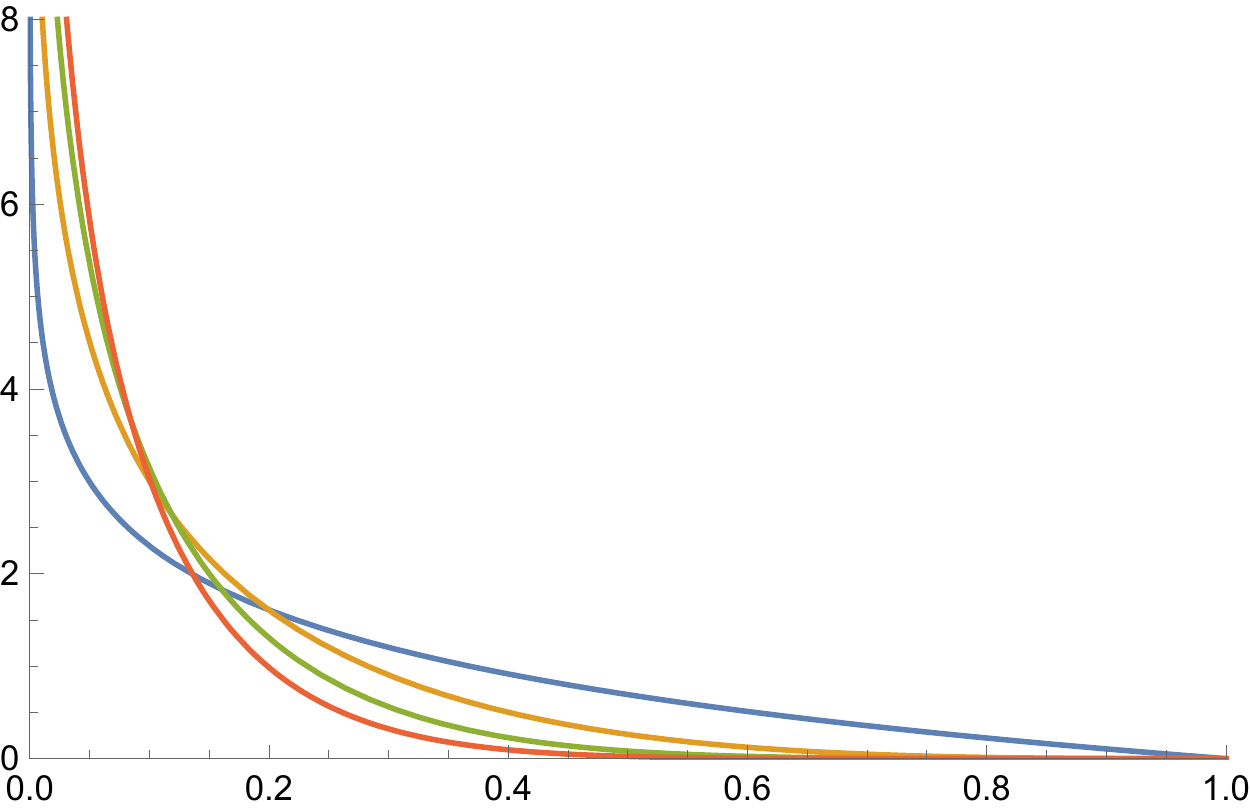}

    \caption{Pdfs $f_\CC(x;k)$ (left) and $f_\RR(x;k)$ (right) for $k=2,4,6,8$}
    \label{fig:pdf_c}
\end{figure}
\end{document}